\documentclass[12pt]{article}
\usepackage[amsmath]{e-jc}

\usepackage{graphicx}

\usepackage{latexsym} 
\usepackage{epsfig}
\usepackage{enumerate}
\usepackage{color}
\usepackage[normalem]{ulem}
\usepackage{booktabs}
\usepackage[table]{xcolor}

\theoremstyle{plain}
\newtheorem{sublemma}{}[theorem]

\newcommand{\cD}{{\mathcal D}}

\newcommand{\cF}{{\mathcal F}}
\newcommand{\cL}{{\mathcal L}}
\newcommand{\cM}{{\mathcal M}}
\newcommand{\cN}{{\mathcal N}}

\newcommand{\cR}{{\mathcal R}}

\newcommand{\cT}{{\mathcal T}} 

\newcommand{\rSPR}{\mathrm{rSPR}}
\newcommand{\SNPR}{\mathrm{SNPR}}

\newcommand{\tc}{{\rm tc}}
\newcommand{\AD}{\mathrm{AD}}

\dateline{Mar 18, 2025}{Aug 16, 2025}{TBD}


\Copyright{The authors. Released under the CC BY-ND license (International 4.0).}

\title{Bounding the SNPR distance between two tree-child networks using generalised agreement forests}

\author{Steven Kelk\authornote{1}
\and
Simone Linz\authornote{2}
\and
Charles Semple\authornote{3}
}

\authortext{1}{Department of Advanced Computing Science, Maastricht University, Maastricht, The Netherlands (\email{steven.kelk@maastrichtuniversity.nl}).}
\authortext{2}{School of Computer Science, University of Auckland, Auckland, New Zealand (\email{s.linz@auckland.ac.nz}).}
\authortext{3}{School of Mathematics and Statistics, University of Canterbury, Christchurch, New Zealand (\email{charles.semple@canterbury.ac.nz}).}

\begin{document}

\maketitle

\begin{abstract}
Agreement forests  continue to play a central role in the comparison of phylogenetic trees since their introduction more than 25 years ago. More specifically, they are used to characterise several distances that are based on tree rearrangement operations and related quantifiers of dissimilarity between phylogenetic trees.\ In addition, the concept of agreement forests continues to underlie most advancements in the development of algorithms that exactly compute the aforementioned measures. In this paper, we introduce agreement digraphs, a concept that generalises agreement forests for two phylogenetic trees to two phylogenetic networks. Analogous to the way in which agreement forests compute the subtree prune and regraft distance between two phylogenetic trees but inherently more complex, we then use agreement digraphs to bound the subnet prune and regraft distance between two tree-child networks from above and below and show that our bounds are tight.
\end{abstract}

\section{Introduction}

Phylogenetic trees and networks play an important role in areas of biology that investigate the relationships between biological entities such as species, viruses, and cells. A central task in the analysis of phylogenetic trees and networks is the quantification of the dissimilarity between them. Distances between phylogenetic trees that provide a measure of dissimilarity can be broadly classified into distances that are based on tree vector representations and those based on tree rearrangement operations~\cite{stjohn17}. While many of the former distances can be computed in polynomial time, the latter distances are typically NP-hard to compute. On the positive side and as summarised in Semple and Steel~\cite{semple03}, tree distances that are based on rearrangement operations provide a  framework to explore and traverse spaces of phylogenetic trees since any tree in a given space can be transformed into any other tree of the same space by a finite number of operations such as nearest neighbour interchange (NNI), subtree prune and regraft (SPR), rooted subtree prune and regraft (rSPR), and tree bisection and reconnection (TBR). In particular, the length of a shortest path between two phylogenetic trees in a given space of trees equals the distance between the two trees under the rearrangement operation that underlies the space.

Computing distances that are based on tree rearrangement operations and related dissimilarity measures such as the minimum hybridisation number for two phylogenetic trees~\cite{baroni05} remains an active area of research (e.g.~\cite{kelk24,olver22,shi18}) despite the NP-hardness of the associated optimisation problems. Indeed, recent algorithmic progress facilitates computations that exactly calculate the aforementioned measures for data sets of remarkable size~\cite{vaniersel22,vanwersch22}. Notably, the concept of agreement forests, which was first introduced by Hein et al.~\cite{hein96}, underpins almost all mathematical and algorithmic advances in this area of research. Intuitively, an agreement forest of two binary phylogenetic trees $\cT$ and $\cT'$ is a decomposition of $\cT$ and $\cT'$ into smaller and non-overlapping subtrees that have the same topology in $\cT$ and $\cT'$. Since the introduction of agreement forests almost 30 years ago, different variants of agreement forests have been used to characterise the rSPR distance between two rooted binary phylogenetic trees, the TBR distance between two unrooted binary phylogenetic trees, and the minimum hybridisation number of two rooted binary phylogenetic trees, as well as to establish related NP-hardness results~\cite{allen01,baroni05,bordewich05}. Subsequent work has focussed on generalising agreement forests to collections of phylogenetic trees of arbitrary size that are not necessarily binary and on the  development of fixed-parameter tractable and approximation algorithms~\cite{chen15,kelk20,linz09,olver22,whidden13}. Additional developments in the context of agreement forests include a generalisation of agreement forests to relaxed agreement forests~\cite{relaxed} and the exploitation of agreement forests to establish extremal results on the SPR, rSPR, and TBR distances~\cite{atkins19,ding11}.  Part of the success of agreement forests is due to the fact that they replace the computation of a measure of dissimilarity between two phylogenetic trees $\cT$ and $\cT'$ with the more static computation of an agreement forest such that the sought-after measure equates to the size of an optimal agreement forest for $\cT$ and $\cT'$.\ Moreover, agreement forests  enable rigorous mathematical arguments that operate only on $\cT$ and $\cT'$ without knowing the topology of any intermediate tree that lies on a shortest path between $\cT$ and $\cT'$ in an associated space of trees.

Since phylogenetic trees are somewhat limited in the type of biological processes that they can represent, rooted phylogenetic networks are increasingly being adopted to represent evolutionary relationships between biological entities whose past does not only include divergence events such as speciation but also convergence events such as lateral gene transfer or hybridisation.\ Inspired by tree rearrangement operations, several network rearrangement operations have recently been developed. For example, the subnet prune and regraft (SNPR) operation generalises rSPR to two rooted binary phylogenetic networks~\cite{bordewich17}. An SNPR operation either adds or deletes a reticulation edge (i.e.,~an edge that is directed into a vertex of in-degree two), or prunes and regrafts a subnetwork in the spirit of rSPR. Since the introduction of SNPR, the operation has been implemented in the Python package PhyloX~\cite{janssen24}, used in studies that, for example, reconstruct reassortment networks and involve a search through the space of rooted phylogenetic networks~\cite{markin22,mueller22,mueller20}, and analysed mathematically with regards to the neighbourhood size~\cite{klawitter18} and properties of shortest length SNPR paths between two rooted phylogenetic networks~\cite{klawitter19b}. Other network rearrangement operations for rooted phylogenetic networks~\cite{erdos21,gambette17,janssen21,janssen18}, unrooted phylogenetic networks~\cite{francis17,huber15,huber16,janssen19}, and semi-directed networks~\cite{phd,linz23} have also been developed and analysed.\ Like SNPR, all such operations generalise a  rearrangement operation for phylogenetic trees to phylogenetic networks.\ For excellent summaries of rearrangement operations for phylogenetic networks, we refer the reader to two PhD theses~\cite{janssen-phd,klawitter-phd}.\ Although a slightly modified framework of agreement forests can be used to compute the SNPR distance between a rooted binary phylogenetic tree and a rooted binary phylogenetic network~\cite{klawitter19b}, the question of how to compute the distances  that result from network rearrangement operations  between two arbitrary binary phylogenetic networks remains open.

In this paper, we generalise agreement forests for two rooted binary phylogenetic trees to agreement digraphs for two rooted binary phylogenetic networks $\cN$ and $\cN'$ that capture the commonalities between them. Focussing on the class of tree-child networks~\cite{cardona09} and using this novel framework of agreement digraphs and their extensions (formal definitions are given Section~\ref{sec:prelim}), we  bound the SNPR distance $d_\tc(\cN,\cN')$ between two tree-child networks $\cN$ and $\cN'$ from above and below, where not only $\cN$ and $\cN'$ are tree-child but also each intermediate network  in an associated sequence. Both bounds are tight and within small constant factors of the minimum number $m_\tc(\cN,\cN')$ of edges in $\cN$ and $\cN'$ that are not contained in an embedding of the agreement digraph and an extension,
where the minimum is taken over all agreement digraphs for $\cN$ and $\cN'$ and their extensions. More specifically, the main result of this paper is the following theorem.
\begin{theorem}\label{t:main}
Let $\cN$ and $\cN'$ be two binary tree-child networks on $X$. Then $$\frac 1 2  m_{\tc}(\cN,\cN’)\leq d_{\tc}(\cN,\cN’) \leq m_{\tc}(\cN,\cN’).$$  
\end{theorem}

The problem of finding an extension of an agreement digraph that optimally captures the commonalities between two rooted binary phylogenetic networks $\cN$ and $\cN'$ may, at first sight, appear to be related to the problem of finding a maximum agreement subnetwork that $\cN$ and $\cN'$ have in common~\cite{choy05,jansson04,valiente-book}. However, upon further inspection, it becomes clear that the two problems are fundamentally different since a maximum agreement subnetwork for $\cN$ and $\cN'$ is not necessarily a component of an optimal agreement digraph for $\cN$ and $\cN'$.\ On the other hand, our work is related to that of Klawitter~\cite{klawitter19,klawitter20}, who has developed an alternative generalisation of agreement forests for two rooted binary phylogenetic networks and for two unrooted binary phylogenetic networks. His generalisation for rooted networks gives rise to collections of agreement subgraphs that, in comparison with our generalisation, may have unlabelled leaves of in-degree one or two and that consequently do not resemble phylogenetic networks. In the same  paper, Klawitter established bounds on the SNPR distance $d^*_{\SNPR}(\cN,\cN')$ between two rooted binary phylogenetic networks $\cN$ and $\cN'$ in terms of collections of agreement subgraphs whose number of unlabelled vertices of degree one is minimised. Without going into detail, this minimum number is referred to as $d_\AD(\cN,\cN')$. In particular, Klawitter established the following theorem.
\begin{theorem}\label{t:klawitter}\cite[Corollary 5.5]{klawitter19}
Let $\cN$ and $\cN'$ be two rooted binary phylogenetic networks on $X$. Then $$d_{\AD}(\cN,\cN’)\leq d^*_{\SNPR}(\cN,\cN’) \leq 6d_{\AD}(\cN,\cN').$$  
\end{theorem}

\noindent While Theorem~\ref{t:klawitter} applies to all rooted binary phylogenetic networks, it remains unknown whether or not the bounds are tight. Comparing Theorems~\ref{t:main} and~\ref{t:klawitter}, we note that  $d^*_\SNPR(\cN,\cN')$ in Theorem~\ref{t:klawitter} refers to the minimum number of SNPR operations that are necessary to transform $\cN$ into $\cN'$, while the quantity $d_\tc(\cN,\cN')$ in Theorem~\ref{t:main} does not only count SNPR operations but, additionally, weights them. More precisely, each SNPR operation that adds or deletes a reticulation edge is weighted  one and each SNPR operation that prunes and regrafts a subnetwork is weighted  two. Hence, $d_\tc(\cN,\cN')$ equates to the minimum sum of weights of SNPR operations that are needed to transform $\cN$ into $\cN'$. We discuss this further in the last section of the paper. Lastly, we note that although Klawitter's generalisation and our definition of an agreement digraph differ, both definitions generalise agreement forests in the sense that, when applied to two rooted binary phylogenetic trees, they can be used to exactly compute their rSPR distance (\cite[Proposition 4.2]{klawitter19} and Proposition~\ref{prop:rspr} of the present paper).

The paper is organised as follows. In Section~\ref{sec:prelim}, we present basic notation and terminology for rooted phylogenetic networks. This is followed by an introduction of  the new concepts of phylogenetic digraphs, which is the main definition building up to agreement digraphs,
and their extensions in Section~\ref{sec:new-def}.  Section~\ref{sec:prop} establishes several basic properties of extensions. Subsequently, in Section~\ref{sec:measures}, we introduce the tree-child SNPR distance between two tree-child  networks $\cN$ and~$\cN'$ and a maximum agreement tree-child digraph for $\cN$ and $\cN$. In Section~\ref{sec:bound}, we establish Theorem~\ref{t:main} and show that our bounds are tight before we finish with some concluding remarks in Section~\ref{sec:conclusion}.

\section{Preliminaries}\label{sec:prelim}

This section provides notation and terminology that is used in the remainder of this paper. Throughout the paper, $X$ denotes a non-empty finite set. It is also worth noting that, except for rooted phylogenetic networks, the graphs that we consider in this paper are not necessarily connected. We start by introducing a broad class of directed acyclic graphs and several definitions that apply to this class. Subsequent definitions in this and the next section consider subclasses of these directed acyclic graphs.\\

\noindent{\bf Directed acyclic graphs.} Let $D$ be a directed acyclic graph. We allow parallel edges in $D$ and note that $D$ may have several vertices with in-degree zero. Furthermore, the undirected graph that underlies $D$ may contain more than one connected component. Let $V_D$ denote the vertex set of $D$, and let $E_D$ denote the edge set of $D$. We say that  a vertex $v$ in $V_D$ is a {\it tree vertex} if $v$ has in-degree one and out-degree one or two, and that $v$ is a {\it reticulation} if $v$ has in-degree two and out-degree one. Furthermore, an edge $(u, v)$ in $D$ is a {\em reticulation edge} if $v$ is a reticulation and, otherwise, $(u, v)$ is a {\em tree edge}.  Lastly, for two vertices $u$ and $v$ in $D$, we say that $u$ is a {\it parent} of $v$ and $v$ is a {\it child} of $u$ precisely if there is an edge $(u,v)$ in $D$. \\

\noindent{\bf Phylogenetic networks.} Rooted phylogenetic networks generalise rooted phylogenetic trees to digraphs with underlying (but no directed) cycles. They allow vertices with in-degree greater than one, which represent non-treelike events such as hybridisation, lateral gene transfer, or recombination. Formally, a {\it rooted binary phylogenetic network on $X$} is a connected directed acyclic graph with a single vertex of in-degree zero that satisfies the following properties:
\begin{enumerate}[(i)]
\item the unique root $\rho$ has in-degree zero and  out-degree one,
\item vertices with out-degree zero have in-degree one, and the set of vertices with out-degree zero is $X$, and
\item all other vertices have either in-degree one and out-degree two, or in-degree two and out-degree one.
\end{enumerate}
The vertices of $\cN$ of out-degree zero are called {\it leaves}, and so $X$ is referred to as the {\it leaf set} $\cL(\cN)$ of $\cN$. In keeping with the literature on distances between two phylogenetic networks, we allow parallel edges in rooted binary phylogenetic networks. Since all phylogenetic networks in this paper are rooted and binary, we simply refer to a rooted binary phylogenetic network on $X$ as a {\it phylogenetic network on $X$}. Now, let $\cN$ be a phylogenetic network on $X$. The vertices of out-degree zero, that is the elements in $X$, are called {\em leaves} and $X$ is referred to as the {\em leaf set} of $\cN$. If a phylogenetic network $\cN$ has no reticulations, we call $\cN$ a {\it phylogenetic $X$-tree}. Moreover, if $\cN$ is a phylogenetic $X$-tree and $X$ contains exactly three elements, say $X=\{a,b,c\}$, then we refer to $\cN$ as a {\it triple} if the underlying path joining the root of $\cN$ and $c$ is vertex-disjoint from that joining $a$ and $b$ in which case, $(a, b, c)$ or, equivalently, $(b, a, c)$ denotes this triple.

Now, let $\cN$ and $\cN'$ be two phylogenetic networks on $X$ with vertex and edge sets $V$ and $E$, and $V'$ and $E'$, respectively. We say that $\cN$ is {\it isomorphic} to $\cN'$ if there is a bijection $\varphi: V\rightarrow V'$ such that $\varphi(x)=x$ for all $x\in X$, and $(u, v)\in E$ if and only if $(\varphi(u), \varphi(v))\in E'$ for all $u, v\in V$. If $\cN$ and $\cN'$ are isomorphic, then we write $\cN\cong\cN'$.\\

\noindent{\bf Tree-child networks.} Let $\cN$ be a phylogenetic network on $X$.\ We say that~$\cN$ is {\it tree-child} if each non-leaf vertex has a child that is a tree vertex or a leaf. Moreover, we say that $\cN$ contains a {\it stack} if there exist two reticulations that are joined by an edge and that $\cN$ contains a pair of {\it sibling reticulations} if there exist two reticulations that have a common parent. A phylogenetic network that is not tree-child is shown in Figure~\ref{fig:tree-child}. In studying the mathematics that underlies phylogenetic networks, tree-child networks have been proven to be particularly successful because of their combinatorial properties that are often exploited to gain traction in establishing mathematical results.\ At the same time, these properties are not overly restrictive from a structural perspective in comparison to level-$1$ networks for example whose underlying cycles are pairwise vertex disjoint. For an overview of classes of phylogenetic networks, we refer the reader to Kong et al.~\cite{kong22}. 

\begin{figure}[t]
    \centering
\scalebox{0.87}{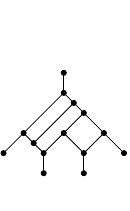}
    \caption{A phylogenetic network on $X=\{1,2,3,4\}$ that is not tree-child because $u$ and $v$ form a stack, and $v$ and $v'$ are a pair of sibling reticulations.}
    \label{fig:tree-child}
\end{figure}

The following well-known equivalence  follows from the definition of a tree-child network and will be freely used throughout the remainder of the paper.

\begin{lemma}
Let $\cN$ be a  phylogenetic network. Then $\cN$ is tree-child if and only if it has no stack, no pair of sibling reticulations, and no pair of parallel edges.
\end{lemma}

The next lemma was established by D\"ocker et al.~\cite[Lemma 7]{doecker21} and shows that the deletion of a reticulation edge of a tree-child network results in another tree-child network.

\begin{lemma}\label{l:reticulation-edge}
Let $\cN$ be a tree-child network on $X$, and let $e=(u,v)$ be a reticulation edge of $\cN$.
Then the  network obtained from $\cN$ by deleting $e$ and suppressing $u$ and~$v$ is a tree-child network on $X$.
\end{lemma}

\section{Phylogenetic digraphs and their extensions}\label{sec:new-def}
In this section, we provide formal definitions of the concepts of phylogenetic digraphs and their extensions. As we will see in Section~\ref{sec:measures}, these definitions generalise agreement forests for two phylogenetic trees to two phylogenetic networks.
We start by providing some high-level ideas that may guide the reader in developing some intuition before providing formal definitions.\ Essentially, a phylogenetic digraph~$\cD$ of a phylogenetic network $\cN$ on $X$ and with root $\rho$ is  a collection of
directed acyclic graphs that each contain at least one element in $X$ with the exception that $\rho$ may be a singleton in $\cD$, whose vertices of out-degree zero are bijectively labelled with the elements in $X$, and for which there is a vertex-disjoint embedding of its components in $\cN$. 
Let $\cM$ be such an embedding of $\cD$ in $\cN$, and let $v$ be a vertex of $\cM$ that either has in-degree zero, or is a reticulation of $\cN$ and has in-degree one and out-degree one. We obtain an extension $\cR$ of $\cM$ by starting at $v$ and extending~$\cM$ towards the root by adding edges of $\cN$ in a certain algorithmic way and then repeating this process for all such vertices of $\cM$. As suggested by the phrase {\it extension of $\cM$}, $\cR$ contains all edges of $\cM$. 
Although $\cD$ may have several embeddings in $\cN$, each embedding is anchored at the leaves of $\cN$ due to the requirement that the leaves of $\cD$ bijectively map to the elements in $X$. As we will see in Section~\ref{sec:prop}, for the purpose of computing the minimum number of edges in $\cN$ that are not contained in an extension relative to a given phylogenetic digraph $\cD$ of $\cN$, where the minimum is taken over all embeddings of $\cD$ in $\cN$ and their extensions, it is sufficient to  consider only a single extension of $\cD$.\\

\noindent {\bf Phylogenetic digraphs.}  Let $D$ be a connected directed acyclic graph, and let $Y$ be a finite set. We say that $D$ is a {\it leaf-labelled  acyclic digraph on $Y$} if one of the following applies:
\begin{enumerate}[(i)]
\item $|Y|=0$ and $D$ is the isolated vertex $\rho$,
\item $|Y|=1$ and $D$ is the isolated vertex labelled with the element in $Y$, or 
\item $|Y|\geq 1$, $D$ has at most one vertex of in-degree zero and out-degree one, in which case this vertex is $\rho$, the leaves of $D$ have in-degree one and out-degree zero and are bijectively labelled with the elements in $Y$, and all other vertices of $D$ have in-degree zero and out-degree two, in-degree one and out-degree two, or in-degree two and out-degree one.
\end{enumerate}
Similar to the leaf set of a phylogenetic network, we refer to $Y$ as the {\it leaf set} of~$D$. Furthermore, $\cL(D)$ denotes the leaf set of $D$. In contrast to a phylogenetic network, a leaf-labelled  acyclic digraph $D$ may have more than one vertex with in-degree zero. Moreover, for a vertex $w$ in $D$ that has two parents $v$ and $v'$, there does not necessarily exist a vertex $u$ such that there are edge-disjoint directed paths from $u$ to $v$ and from $u$ to $v'$ in $D$.

Let $\cN$ be a phylogenetic network on $X$ with root $\rho$, and let $D$ be a  leaf-labelled  acyclic digraph on $Y$ with $Y\subseteq X$. Recall that $D$ may or may not contain a vertex~$\rho$ with in-degree zero and out-degree one. We say that $\cN$ {\it displays} $D$ if there exists a subgraph of $\cN$ that is isomorphic to $D$ up to suppressing vertices with in-degree one and out-degree one, in which case  we call the subgraph  an {\it embedding} $M$ of $D$ in $\cN$ and view the edge set of $M$ as a subset of the edge set of $\cN$. More generally, for a collection $\cD=\{D_1, D_2, \ldots, D_k\}$ of leaf-labelled  acyclic digraphs, we say that $\cN$ {\it displays} $\cD$ if there exists an embedding $M_i$ of $D_i$ in $\cN$ for each $i\in\{1,2,\ldots,k\}$ such that $M_{j}$ and $M_{j'}$ are vertex disjoint for all distinct $j,j'\in\{1,2,\ldots,k\}$, in which case we refer to $\cM=\{M_1,M_2,\ldots,M_k\}$ as an {\it embedding} of $\cD$ in $\cN$. Now let~$\cM$ be an embedding of $\cD$ in $\cN$.\ Recalling that we allow tree vertices in $\cM$ to have in-degree and out-degree one, we say that  $\cM$ is {\it tree-child} if each non-leaf vertex of  $\cM$ has a child that is a tree vertex or a leaf.

Let $\cN$ be a phylogenetic network on $X$ with root $\rho$. Let $\cD=\{D_\rho,D_1, D_2, \ldots, D_k\}$ be a collection of  leaf-labelled  acyclic digraphs. Then $\cD$ is called a {\it phylogenetic digraph} of $\cN$ if the following three properties are satisfied:
\begin{enumerate}[(i)]
\item  the leaf sets $\cL(D_\rho),\cL(D_1), \cL(D_2), \ldots,\cL(D_k)$ partition $X$ and $D_\rho$ is the only element in $\cD$ that contains $\rho$,
\item $\rho$ is either an isolated vertex in $\cD$ or the unique vertex in $\cD$ with in-degree zero and out-degree one, and
\item there exists an embedding $\cM=\{M_\rho,M_1,M_2,\ldots,M_k\}$  of $\cD$ in $\cN$.
\end{enumerate}
Lastly, a phylogenetic digraph $\cD$ of $\cN$ is called a {\it tree-child digraph} of $\cN$ if each non-leaf vertex of $\cD$ has a child that is a leaf or a tree vertex.\\

\noindent {\bf Extensions and root extensions.} Let $\cN$ be a phylogenetic network on $X$. Furthermore, let  $\cM=\{M_\rho,M_1,M_2,\ldots,M_k\}$ be an embedding of a phylogenetic digraph $\cD=\{D_\rho,D_1, D_2, \ldots, D_k\}$ of $\cN$. We obtain an {\it  extension} $\cR$ of $\cD$ in $\cN$ from~$\cM$ by initially setting $\cR=\cM$ and then repeatedly applying one of the following two operations until no further such operation is possible:
\begin{enumerate}[(E1)]
\item For a vertex $v$ of $\cR$ with in-degree zero, add $(u,v)$ to $\cR$ if $u\notin \cR$.
\item For a vertex $v$ of $\cR$ with in-degree one and out-degree one and $v$ being a reticulation in $\cN$, add $(u,v)$ to $\cR$ if $u\notin \cR$.
\end{enumerate}
By construction, observe that $\cR$ contains exactly $k+1$ connected components and that there is a natural bijection between these components and the components in~$\cD$. We therefore set $\cR=\{R_\rho,R_1,R_2,\ldots,R_k\}$ and call $R_i$ an {\it extension} of $D_i$ in~$\cN$ for each $i\in\{\rho,1,2,\ldots,k\}$.
It follows from the construction of $\cR$ that there is no vertex of  out-degree two in $\cR$ that is not also a vertex of  out-degree two in $\cM$. Moreover, by construction, any underlying cycle in $\cR$ is also an underlying cycle in~$\cM$. Lastly, we define  $\cR$ to be {\it tree-child} precisely if $\cM$ is tree-child, and refer to~$\cM$ as the embedding of $\cD$ that {\it underlies} $\cR$.

We next introduce a special type of extension. We call an  extension $\cR$ of $\cD$ in~$\cN$ a {\it root extension} of $\cD$ in $\cN$ if it can be obtained from $\cM$ by initially setting $\cR=\cM$ and then repeatedly applying (E1) only until no further such operation is possible. Now let $\cR=\{R_\rho,R_1,R_2,\ldots,R_k\}$ be a root extension of $\cD$ in $\cN$. Similar to the terminology for an extension, we call $R_i$ a {\it root extension} of $D_i$ in $\cN$ for each $i\in\{\rho,1,2,\ldots,k\}$. Let $r$ be a vertex of in-degree zero and out-degree zero or two in $\cD$, and let $P$ be the unique maximal length directed path in $\cR$ that starts at a vertex $u$ of in-degree zero and ends at $r$. We refer to $P$ as the {\it root path} of $r$. Note that $P$ may have no edge in which case $u=r$. If $u\ne r$, then $u$ has in-degree zero and out-degree one in $\cR$. Figure~\ref{fig:example} illustrates the concepts of phylogenetic digraphs, extensions, and root extensions.

\begin{figure}[t]
\centering
\scalebox{0.87}{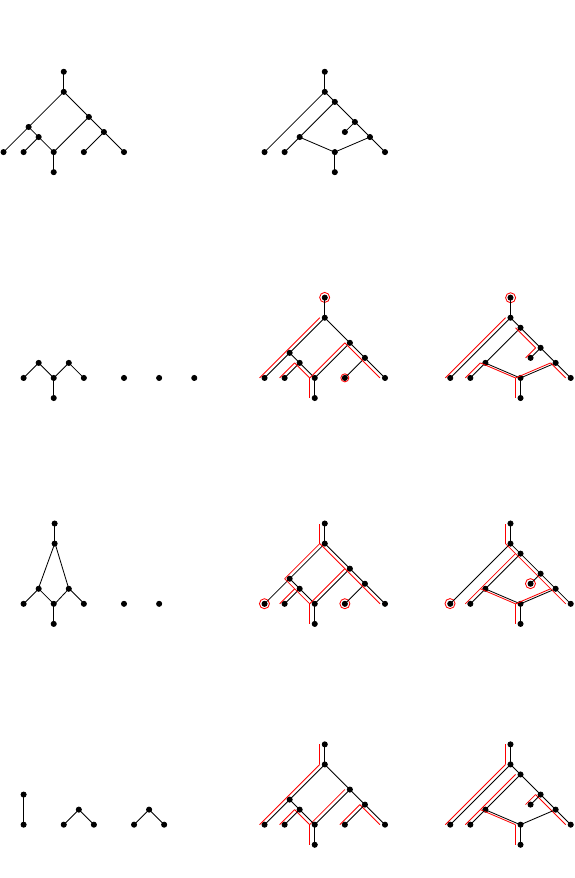}
\caption{Two phylogenetic networks $\cN$ and $\cN'$ and three phylogenetic digraphs $\cD_1$, $\cD_2$, and $\cD_3$ for $\cN$ and $\cN'$.\ For each $i\in\{1,2,3\}$, $\cR_i$ is an extension of $\cD_i$ in $\cN$ and $\cR_i'$ is an extension of $\cD_i$ in $\cN'$, where the edges and vertices of the extensions are indicated  in red. Note that $\cR_3$ is not a root extension of $\cD_3$ in $\cN$.}
\label{fig:example}
\end{figure}

To ease reading throughout the remainder of the paper, we often consider a phylogenetic network $\cN$, a phylogenetic digraph $\cD$ of $\cN$, and an extension $\cR$ of $\cD$ in~$\cN$. In this case, we view the vertex and edge set of $\cR$ (as well as the vertex and edge set of any embedding of $\cD$ in $\cN$) as a subset of the vertex and edge set of $\cN$, respectively. Furthermore, for clarity, the in-degree (resp.\ out-degree) of a vertex~$v$ in $\cR$ refers to the number of edges in $\cR$ that are directed into (resp.\ out of) $v$. Lastly, let $\cD=\{D_\rho,D_1,D_2,\ldots,D_k\}$ be a phylogenetic digraph of a phylogenetic network $\cN$. Then there exists an embedding $\cM=\{M_\rho,M_1,M_2,\ldots,M_k\}$ of $\cD$ in~$\cN$ such that $M_i$ and $M_j$ are vertex disjoint for all distinct $i,j\in\{\rho,1,2,\ldots,k\}$ and an extension $\cR=\{R_\rho,R_1,R_2,\ldots,R_k\}$ of $\cD$ in $\cN$ such that $R_i$ and $R_j$ are vertex disjoint for all distinct $i,j\in\{\rho,1,2,\ldots,k\}$. It follows that each edge in $\cD$ corresponds to a unique directed path in $\cM$ (resp.\ $\cR$) whose non-terminal vertices all have in-degree one and out-degree one in $\cM$ (resp.\ $\cR$), and each vertex in $\cD$ corresponds to a unique vertex in $\cM$ (resp.\ $\cR$). Reversely, each edge in $\cM$  corresponds to a unique edge in $\cD$. We will freely use this correspondence throughout the paper.

\section{Properties of extensions}\label{sec:prop}

In this section we establish several results for extensions that will be useful in the subsequent sections. Let $\cD$ be a phylogenetic digraph  of a phylogenetic network $\cN$. Given the algorithmic definition of an extension, different orderings of the elements in $\cD$ may result in different extensions even for  a fixed underlying embedding. Also, if $\cR$ and $\cR'$ are extensions of $\cD$ in $\cN$ with distinct underlying embeddings, then $\cR$ and $\cR'$ are different. 

\begin{lemma}\label{l:root-ext-basic}
Let $\cD$ be a phylogenetic digraph of a phylogenetic network $\cN$ on $X$ with root $\rho$, and let $\cR$ be an extension of $\cD$ in $\cN$. Then the following hold:
\begin{enumerate}[{\rm (i)}]
\item If $v$ is a  vertex in $\cN$ that is not in $X\cup\{\rho\}$, then there is an edge $(v,w)$ in $\cN$ that is in $\cR$.
\item Each vertex in $\cN$ is contained in $\cR$. 
\end{enumerate}
\end{lemma}

\begin{figure}[t]
\center
\scalebox{0.87}{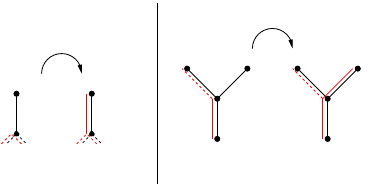}
\caption{Assuming that $v\notin \cR$, the setup as described in the proof of Lemma~\ref{l:root-ext-basic}(ii) for when (I) $w$ is a tree vertex or a leaf, and (II) $w$ is a reticulation in $\cN$. In both cases, an application of (E1) or (E2) can be used to extend $R_i$ by an additional edge so that the resulting extension contains $v$. Red solid lines indicate vertices and edges of $R_i$. Black (resp. red) dashed lines indicate vertices and edges of $\cN$ (resp. $R_i$) that may or may not be vertices and edges of  $\cN$ (resp. $R_i$).}
\label{fig:4-1-ii}
\end{figure}

\begin{proof}
Let $\cD=\{D_\rho,D_1,D_2,\ldots,D_k\}$, and let $\cR=\{R_\rho,R_1,R_2,\ldots,R_k\}$. To see that (i) holds, recall that $\cR$ does not contain any vertex with out-degree zero that is not in $X\cup\{\rho\}$. To complete the proof, we establish (ii). This part of the proof is illustrated in Figure~\ref{fig:4-1-ii}. By definition of $\cD$, the root and each leaf of $\cN$ is contained in $\cR$.  Towards a contradiction, we may therefore assume that there is a tree vertex or a reticulation in $\cN$ that is not contained in $\cR$. Let $v$ be a vertex of $\cN$ that is not in $\cR$ such that every vertex that is distinct from $v$ and lies on a directed path from~$v$ to a leaf in $\cN$ is in $\cR$. Furthermore, let $w$ be a child of $v$. If $w$ is a tree vertex or leaf in $\cN$, then there exists a component $R_i$ with $i\in\{\rho,1,2,\ldots,k\}$ such that $w$ has in-degree zero in $R_i$, thereby contradicting that $R_i$ is an extension of $D_i$ as we can apply (E1). Otherwise, if $w$ is a reticulation in $\cN$, then it follows from (i) that there exists a component $R_i$ with $i\in\{\rho,1,2,\ldots,k\}$ such that $w$ has in-degree zero, or in-degree one and out-degree one in $R_i$, thereby again contradicting that~$R_i$ is an extension of $D_i$ as we can apply either (E1) or (E2), respectively. 
\end{proof}

\begin{lemma}\label{l:root-ext-tree-vertex}
Let $\cN$ be a phylogenetic network on $X$. Furthermore, let $\cM$ be an embedding of a phylogenetic digraph $\cD$ of $\cN$, and let $\cR$ be an extension of $\cD$ such that $\cM$ underlies $\cR$. Then the following hold for each tree vertex $v$ in $\cN$.
\begin{enumerate}[{\rm (i)}]
\item If $v$ is in $\cM$, then no edge directed out of $v$ is in $\cR$ and not in $\cM$.
\item If $v$ is not in $\cM$, then exactly one of the two edges directed out of $v$ is contained in $\cR$.
\end{enumerate}
\end{lemma}

\begin{proof}
Let $u$ be the parent of $v$, and let $w$ and $w'$ be the two children of $v$ in $\cN$.  Furthermore, let $\cR=\{R_\rho,R_1,R_2,\ldots,R_k\}$. We first show that (i) holds. Since $\cM$ does not contain any vertex of out-degree zero that is not in $X\cup\{\rho\}$, it follows that at least one of $(v,w)$ and $(v,w')$ is in $\cM$. Without loss of generality, we may assume that $(v,w)\in \cM$. The setup of the proof is shown in Figure~\ref{fig:4-2}(i). If $(v,w')\in \cM$, then the result clearly holds. On the other hand, if $(v,w')\notin \cM$, then it follows from the definition of an extension that  $(v,w')\notin \cR$, thereby establishing (i). We now turn to (ii) which is illustrated in Figure~\ref{fig:4-2}(ii). It follows from Lemma~\ref{l:root-ext-basic}(ii) that $v\in \cR$. Since $\cR$ does not contain any vertex with out-degree zero that is not in $X\cup\{\rho\}$, at least one of $(v,w)$ and $(v,w')$ is contained in $\cR$. Moreover, again by the definition of an extension, at most one of $(v,w)$ and $(v,w')$ is contained in $\cR$. The lemma now follows.
\end{proof}

\begin{figure}[t]
\center
\scalebox{0.87}{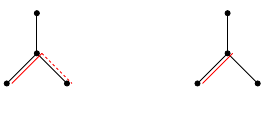}
\caption{Setup as described in the proof of parts (i) and (ii) of Lemma~\ref{l:root-ext-tree-vertex}. The red solid line in (i) (resp. (ii)) indicates that $(v,w)$ is an edge of $\cM$ (resp. $\cR$), and the red dashed line in (i) indicates that $(v,w')$ may or may not be an edge of $\cM$.}
\label{fig:4-2}
\end{figure}

Let $\cD$ be a phylogenetic digraph of a phylogenetic network $\cN$ on $X$. The next proposition shows that the number of edges that are in $\cN$ but not in an extension of $\cD$ does not depend on the extension.

\begin{proposition}\label{p:equal-size}
Let $\cN$ be a phylogenetic network on $X$, and let $\cR$ and $\cR'$ be two extensions of a phylogenetic digraph $\cD$ of $\cN$. Then $$|E_\cN-E_\cR|=|E_\cN-E_{\cR'}|.$$ 
\end{proposition}

\begin{proof}
Let $\cM$ and $\cM'$ be the embeddings of $\cD$ in $\cN$ that underlie $\cR$ and $\cR'$, respectively. First assume that $\cM=\cM'$. Since $|E_\cN-E_\cM|=|E_\cN-E_{\cM'}|$ and each iteration of (E1) or (E2) in the construction of $\cR$ and $\cR'$ either adds a new edge $(v,w)$ such that  $w$ is already in $\cR$  and $v$ is not already in $\cR$ or a new edge $(v',w')$ such that  $w'$ is already in $\cR'$  and $v'$ is not already in $\cR'$, the result follows from Lemma~\ref{l:root-ext-basic}(ii). Second assume that $\cM\ne\cM'$. Consider an edge $e=(v,w)$ of $\cN$. If $v=\rho$, then, as $\rho$ has out-degree one in $\cN$, either $e\in\cR$ and $e\in \cR'$, or $e\notin \cR$ and $e\notin \cR'$. Furthermore, if $v$ is a reticulation in $\cN$, then it follows from Lemma~\ref{l:root-ext-basic}(i) that $e\in \cR$ and  $e\in \cR'$. We next consider all edges in $\cN$ that are directed out of a tree vertex. Let $v$ be a tree vertex of $\cN$, and let $e=(v,w)$ and $e'=(v,w')$ be the two edges directed out of $v$. We next consider four cases that will subsequently be used for a counting argument.
\begin{enumerate}[(1)]
\item Suppose that both of $e$ and $e'$ are contained in one of $\cM$ and $\cM'$, and  neither $e$ nor $e'$ is contained in the other embedding.
Without loss of generality, we may assume that both of $e$ and $e'$ are contained in $\cM$ and, therefore, in $\cR$. As~$\cM'$ does not contain a vertex with out-degree zero that is not in $X\cup\{\rho\}$, it follows that $v\notin \cM'$. Thus, by Lemma~\ref{l:root-ext-tree-vertex}(ii), exactly one of $e$ and $e'$ is in $\cR'$.
\item Suppose that exactly one of $e$ and $e'$ is contained in one of $\cM$ and $\cM'$ and neither $e$ nor $e'$ is contained in the other embedding.
Without loss of generality, we may assume that $\cM$ contains $e$ and does not contain $e'$, and that $\cM'$ contains neither $e$ nor $e'$. Evidently $e\in \cR$ and, by the definition of an extension, $e'\notin \cR$. Moreover, as $v\notin \cM'$ it follows from Lemma~\ref{l:root-ext-tree-vertex}(ii) that exactly one of $e$ and $e'$ is in $\cR'$. 
\item Suppose that exactly one of $e$ and $e'$ is contained in each of  $\cM$ and $\cM'$.
Again by Lemma~\ref{l:root-ext-tree-vertex}(i), exactly one of $e$ and $e'$ is contained in each of $\cR$ and  $\cR'$.
\item Suppose that both of $e$ and $e'$ are contained in one of $\cM$ and $\cM'$, and exactly one of $e$ and $e'$ is contained in the other embedding.
Without loss of generality, we may assume that both of $e$ and $e'$ are contained in $\cM$ and, therefore, in $\cR$. Then, by Lemma~\ref{l:root-ext-tree-vertex}(i), exactly one of $e$ and $e'$ is contained in $\cR'$.
\end{enumerate}
It follows that, if Case (2) or (3) applies to $v$, then exactly one of $e$ and $e'$ is an element in $E_\cN-E_\cR$ and exactly one of $e$ and $e'$ is an element in $E_\cN-E_{\cR'}$. We next turn to Cases (1) and (4).
Let $V$ (resp.\ $V'$) denote the set that contains precisely each tree vertex of $\cN$ whose two outgoing edges are both in $\cM$ (resp.\ $\cM'$). Since $\cM$ and $\cM'$ are embeddings of $\cD$ in $\cN$, we have $|V|=|V'|$ and, consequently, $|V-V'|=|V'-V|$. Hence, the number of tree vertices in $\cN$ for which both outgoing edges are in $\cR$ and exactly one outgoing edge is in $
\cR'$ is equal to the number of tree vertices in $\cN$ for which both outgoing edges are in $\cR'$ and exactly one outgoing edge is in $\cR$. 
This completes the proof of the proposition.
\end{proof}

The last proposition motivates the following terminology. Let $\cD$ be a phylogenetic digraph of a phylogenetic network $\cN$ on $X$, and let $\cR$ be an extension of $\cD$ in $\cN$. We set $c_\cD=|E_\cN-E_\cR|$ and refer to $c_\cD$ as the {\it cut size} of $\cD$ in $\cN$. By Proposition~\ref{p:equal-size}, $c_\cD$ is well defined. 

The next two results consider cut sizes of root extensions in tree-child networks.

\begin{lemma}\label{l:extension-equiv}
Let $\cD$ be a phylogenetic digraph of a tree-child network $\cN$ on $X$, and let $\cR$ be an extension of $\cD$ in $\cN$ with cut size $c_\cD$. Then there also exists a root extension of $\cD$ in $\cN$ with cut size $c_\cD$.
\end{lemma}

\begin{proof}
Let $\cM$ be the embedding of $\cD$ in $\cN$ that underlies $\cR$. If $\cR$ is not a root extension of $\cD$ in $\cN$, then there exists a reticulation edge $(u,v)$ in $\cN$ such that $(u,v)\in\cR$ and $(u,v)\notin\cM$. Since $\cN$ is tree-child, $u$ is a tree vertex. Let $v'$ be the child of $u$ in $\cN$ with $v\ne v'$. Clearly, $v'$ is a tree vertex or a leaf in $\cN$, and $(u,v')\notin \cR$. Let $R_v$ be the component in $\cR$ that contains $v$, and let $R_{v'}$ be the component in $\cR$ that contains $v'$. If $R_v$ and $R_{v'}$ are distinct, obtain $R'_v$ from $R_v$ by deleting each vertex $s$ for which there exists a directed path from $s$ to $u$, and obtain  $R'_{v'}$ from $R_{v'}$ by adding $(u,v')$ and each edge of $R_v$ that is contained in a directed path ending at $u$ and, if $R_v$ and $R_{v'}$ are not distinct, obtain $R'_v$ from $R_v$ by deleting $(u, v)$ and adding $(u, v')$.
Then $\cR'=(\cR-\{R_v,R_{v'}\})\cup\{R_v',R_{v'}'\}$ 
is an extension of $\cD$ in $\cN$ with cut size $c_\cD$. It is straightforward to check that $\cM$ is the embedding of $\cD$ in $\cN$ that underlies $\cR'$ and that $\cR'$ has one reticulation edge less than $\cR$. If $\cR'$ is not a root extension  of $\cD$ in $\cN$, then repeat the construction for a reticulation edge in $\cN$ that is in $\cR'$ but not in $\cM$ until no such edge remains. This completes the proof of the lemma.
\end{proof}

\begin{figure}[t]
    \centering
\scalebox{0.87}{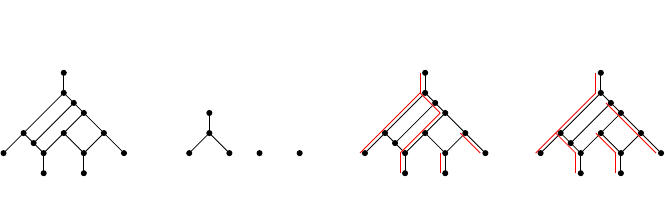}
    \caption{A phylogenetic network $\cN$, a phylogenetic digraph $\cD$ of $\cN$, and two root extensions $\cR$ and $\cR'$ (indicated by red lines) of $\cD$ in $\cN$ with $|E_\cN-E_\cR|=6\ne5=|E_\cN-E_{\cR'}|$.}
    \label{fig:cut}
\end{figure}

\noindent We end this section by noting that Proposition~\ref{p:equal-size} does not hold for two root extensions of  arbitrary phylogenetic networks (see Figure~\ref{fig:cut} for an example). Nevertheless, we have the following result for tree-child networks. Its proof is similar to that of Proposition~\ref{p:equal-size} and is omitted.

\begin{proposition}\label{p:equal-size-tc}
Let $\cN$ be a tree-child network on $X$, and let $\cR$ and $\cR'$ be two root extensions of a phylogenetic digraph $\cD$ of $\cN$. Then $$|E_\cN-E_\cR|=|E_\cN-E_{\cR'}|.$$ 
\end{proposition}

\section{Measures of dissimilarity between two tree-child networks}\label{sec:measures}

In this short section, we formally define the tree-child SNPR distance between two tree-child networks and a measure that is associated with an extension of a phylogenetic digraph common to two tree-child networks. This measure bounds the SNPR distance between two tree-child networks from above and below.

Let $\cN$ and $\cN'$ be two tree-child networks on $X$. Let $\cD=\{D_\rho,D_1, D_2, \ldots, D_k\}$ be a collection of leaf-labelled acyclic digraphs. If $\cD$ is a phylogenetic digraph of $\cN$ and $\cN'$, we refer to $\cD$ as an {\it agreement digraph} of $\cN$ and $\cN'$. Now let $\cD$ be an agreement digraph for $\cN$ and $\cN'$. If $\cD$ is tree-child, we say that $\cD$ is an {\it  agreement tree-child  digraph} for $\cN$ and $\cN'$. In the remainder of this paper, we are particularly interested in agreement digraphs of $\cN$ and $\cN'$ whose cut size is minimum. To this end, let $\cD$ be an agreement tree-child digraph for $\cN$ and $\cN'$, and let $c_\cD$ and~$c'_\cD$ be the cut size of $\cD$ in $\cN$ and $\cN'$, respectively. 
Then $\cD$ is called a {\it maximum agreement tree-child  digraph} for $\cN$ and $\cN'$ if the sum $c_\cD+ c'_\cD$  is minimum over all agreement tree-child digraphs for $\cN$ and $\cN'$, in which case $m_\tc(\cN,\cN')$ denotes this minimum number.
To calculate $m_\tc(\cN,\cN')$, it follows from Proposition~\ref{p:equal-size} that it is sufficient to consider a single  extension of each agreement digraph for $\cN$ and $\cN'$. 
Referring back to Figure~\ref{fig:example}, observe that each of the three phylogenetic digraphs $\cD_1$, $\cD_2$, and $\cD_3$ is in fact an agreement digraph for the two tree-child networks $\cN$ and $\cN'$ that are shown in the same figure.

Let $\cN$ and $\cN'$ be two tree-child networks,  let $\cD$ be an agreement tree-child digraph for $\cN$ and $\cN'$, and let $\cR$ and $\cR'$  be an extension of  $\cD$ in $\cN$ and $\cN'$, respectively. We note that, similar to the elements of an agreement forest,  the elements in $\cD$ can be embedded in $\cN$ and $\cN'$. Intuitively, they can be thought of as subnetworks that are common to $\cN$ and $\cN'$. On the other hand, the digraphs induced by the edges in $E_\cR-E_\cM$ and $E_{\cR'}-E_{\cM'}$, where $\cM$ and $\cM'$ is the embedding that underlies $\cR$ and $\cR'$, respectively, are not necessarily the same. Although each connected component in such a digraph is a rooted tree whose (unique) root is a vertex of $\cM$ and $\cM'$, respectively, and whose edges are directed {\it towards} the root, one digraph may contain directed rooted trees with a small total number of unlabelled leaves and the other one may contain directed rooted trees with a much larger total number of unlabelled leaves.

Now, let $\cN$ be a phylogenetic network on $X$, and let $e=(u, v)$ be an edge in $\cN$. We consider the following three operations applied to $\cN$:
 
\begin{enumerate}[]
\item {\bf SNPR$^\pm$} If $u$ is a tree vertex, then delete $e$, suppress $u$, subdivide an edge that is not a descendant of $v$ with a new vertex $u'$, and add the new edge $(u', v)$.
\item {\bf SNPR$^-$} If $u$ is a tree vertex and $v$ is a reticulation, then delete $e$, and suppress $u$ and $v$.
\item{\bf SNPR$^+$} Subdivide $e$ with a new vertex $v'$, subdivide an edge in the resulting network that is not a descendant of $v'$ with a new vertex $u'$, and add the new edge $(u', v')$.
\end{enumerate}

\noindent By definition of a tree vertex, $u\ne\rho$ if we apply an SNPR$^\pm$.  If it is not important which of  SNPR$^-$, SNPR$^+$, and SNPR$^\pm$ has been applied to $\cN$ we simply refer to it as an {\it $\SNPR$}. As observed by Bordewich et al.~\cite[Proposition 3.1]{bordewich17}, the operation is {\it reversible}, i.e. if $\cN'$ is a phylogenetic network on $X$ that can be obtained from~$\cN$ by a single SNPR, then $\cN$ can also be obtained from $\cN'$ by a single SNPR. Lastly, we note that the well-known rSPR operation is an application of SNPR$^\pm$ to a phylogenetic tree.

Let $\cN$ and $\cN'$ be two phylogenetic networks on $X$. An {\it $\SNPR$ sequence} $\sigma$ for $\cN$ and $\cN'$ is a sequence $$\sigma=(\cN=\cN_0, \cN_1, \cN_2, \ldots, \cN_t=\cN')$$ of phylogenetic networks on $X$ such that, for all $i\in \{1, 2, \ldots, t\}$, we have $\cN_i$ is obtained from $\cN_{i-1}$ by a single $\SNPR$ in which case, we say that $\sigma$ {\em connects} $\cN$ and~$\cN'$. We refer to $t$ as the {\it length} of $\sigma$. Let $t^\pm$ be the number of phylogenetic networks in $\{\cN_1,\cN_2,\ldots,\cN_t\}$ that have been obtained by an SNPR$^\pm$  and, similarly, let $t^-$ and $t^+$ be the number of phylogenetic networks in $\{\cN_1,\cN_2,\ldots,\cN_t\}$ that have been obtained by an SNPR$^-$ and SNPR$^+$, respectively. Clearly, $t^\pm+t^-+t^+=t$. We set $$w(\sigma)=2t^\pm+t^-+t^+$$ and refer to $w(\sigma)$ as the {\it weight} of $\sigma$. Intuitively, each deletion or addition of an edge contributes one to $w(\sigma)$.\ Thus, since an SNPR$^\pm$ deletes an edge and, subsequently adds a new edge, such an operation adds two to $w(\sigma)$. It was shown by Bordewich~et~al.~\cite[Proposition 3.2]{bordewich17} that any two phylogenetic networks $\cN$ and~$\cN'$ on the same leaf set are connected by an $\SNPR$ sequence. If $\cN$ and $\cN'$ are tree-child, then, by the same proposition, there also exists an $\SNPR$ sequence that connects $\cN$ and $\cN'$ such that each network in the sequence is tree-child. We refer to such an SNPR sequence as a {\it tree-child $\SNPR$ sequence}. Moreover, we define 
the  {\em tree-child $\SNPR$ distance} $d_\tc(\cN,\cN')$ between two tree-child networks $\cN$ and $\cN'$ as the minimum weight of any tree-child $\SNPR$ sequence connecting $\cN$ and $\cN'$. If $\cN$ and $\cN'$ are two phylogenetic $X$-trees, then $d_\rSPR(\cN,\cN')$ denotes the minimum number of rSPR operations needed to transform $\cN$ into $\cN'$.\\

\noindent{\bf Global assumption.} Let $\sigma=(\cN_0, \cN_1, \cN_2, \ldots, \cN_t)$  be an SNPR sequence that connects the two phylogenetic networks $\cN_0$ and $\cN_t$ on $X$. Throughout the remainder of the paper, we assume that there exists no $i\in\{1,2,\ldots,t\}$, such that $\cN_i$ can be obtained from $\cN_{i-1}$ by an SNPR$^\pm$ that deletes a reticulation edge in $\cN_{i-1}$. Indeed, if such an $i$ exists, then there exists an SNPR sequence $$\sigma'=(\cN_0,\cN_1,\cN_2,\ldots,\cN_{i-1},\cN_{i}',\cN_i,\cN_{i+1},\ldots,\cN_t)$$ such that $\cN_{i}'$ can be obtained from $\cN_{i-1}$ by an SNPR$^-$ and $\cN_{i}$ can be obtained from $\cN_{i}'$ by an SNPR$^+$.  Since we are interested in SNPR sequences of minimum weight and  $w(\sigma')=w(\sigma)$, no generality is lost.  

\section{Bounding the tree-child SNPR distance}\label{sec:bound}

In this section, we establish Theorem~\ref{t:main} and show that the bounds are tight.
Let $\cN$ and $\cN’$ be two tree-child networks, and let $\cR$ be an extension of a tree-child digraph $\cD$ of $\cN$. We first show that, if $\cN$, $\cN’$, and $\cR$ satisfy certain properties, then there exists an extension of $\cD$ in $\cN'$ such that the cut sizes of $\cD$ in $\cN$ and $\cN’$ differ by at most one.

\begin{lemma}\label{l:mod-ext}
Let $\cN$ and $\cN'$ be two tree-child networks on $X$,  let $\cD$ be a tree-child digraph of $\cN$, and let $\cR$ be an extension of $\cD$ in $\cN$. 
\begin{enumerate}[{\rm (i)}]
\item  Let $(u,v)$ be a reticulation edge of $\cN$ such that $(u,v)\notin\cR$. If $\cN'$ can be obtained from $\cN$ by an {\em SNPR$^-$} that deletes $(u,v)$ and suppresses $u$ and $v$, then $\cD$ is a tree-child digraph of $\cN'$  and there exists an extension $\cR'$ of $\cD$ in~$\cN'$ such that $$|E_{\cN}-E_{\cR}|-1=|E_{\cN'}-E_{\cR'}|.$$
\item Let  $e=(u,w)$ and $e'=(u',w')$ be two distinct tree edges of $\cN$. If $\cN'$ can be obtained from $\cN$ by an {\em SNPR$^+$} that subdivides $e$ and $e'$ with a new vertex $v$ and $v'$, respectively, and adds the new reticulation edge $(v',v)$, then $\cD$ is a tree-child digraph of $\cN'$  and there exists an extension $\cR'$ of $\cD$ in $\cN'$ such that $$|E_{\cN}-E_{\cR}|+1=|E_{\cN'}-E_{\cR'}|.$$
\item Let $e=(u,v)$ be a tree edge of $\cN$ such that $e\notin\cR$, and let $f=(p_{u'},c_{u'})$ be an edge in $\cN$ that is distinct from $e$. If $\cN'$ can be obtained from $\cN$ by an {\em SNPR$^\pm$} that deletes $e$, suppresses $u$, subdivides $f$ with a new vertex $u'$, and adds the new edge $(u',v)$, then $\cD$ is a tree-child digraph of $\cN'$  and there exists an extension $\cR'$ of $\cD$ in $\cN'$ such that $$|E_{\cN}-E_{\cR}|=|E_{\cN'}-E_{\cR'}|.$$
\end{enumerate}
\end{lemma}

\begin{figure}[t]
\centering
\scalebox{0.87}{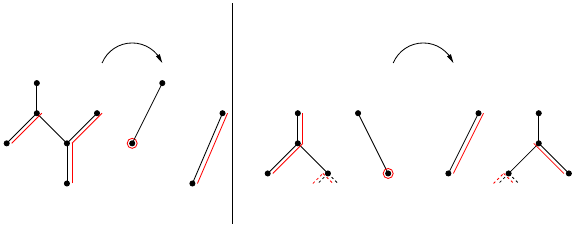}
\caption{Setup of parts (i) and (iii) in the proof of Lemma~\ref{l:mod-ext}. Red lines indicate vertices and edges in $\cR$ and $\cR'$. In (i), observe that $u$ has in-degree zero and out-degree one, and $v$ has in-degree one and out-degree one in $\cR$ whereas, in (iii), observe that $u$ has in-degree one and out-degree one, $(p_{u'},c_{u'})\notin\cR$, and $v$ may or may not be a leaf in $\cN$. Vertex labels $y_1$, $y_2$, $y_3$, and $y_4$ are only used to clarify the figure.}
\label{fig:lemma-6-1}
\end{figure}

\begin{proof}
The setup used to establish parts (i) and (iii) is illustrated in Figure~\ref{fig:lemma-6-1}. We first establish (i). Since $(u,v)\notin\cR$ and $\cD$ is a tree-child digraph of $\cN$, it follows that  $\cD$ is also such a digraph of $\cN'$. Now, let $\cR'$ be the digraph obtained from $\cR$ by applying the following operation to each vertex $w\in\{u,v\}$. If $w$ has in-degree zero  and out-degree one in $\cR$, then delete $w$ and, if $w$ has in-degree one and out-degree one in $\cR$, then suppress $w$.  As $(u,v)\not\in \cR$, it follows by Lemma~\ref{l:root-ext-basic} that each of $u$ and $v$ has either in-degree zero and out-degree one, or in-degree one and out-degree one in $\cR$. Thus, $\cR'$ is well defined. It now follows from the construction that $\cR'$ is an extension of $\cD$ in $\cN'$ with $|E_{\cN}-E_{\cR}|-1=|E_{\cN'}-E_{\cR'}|$.

To see that (ii) holds, observe first that since $\cD$ is a tree-child digraph of $\cN$, it follows from the construction of $\cN'$ from $\cN$ that $\cD$ is also such a digraph of $\cN'$. Reversing the construction of $\cR'$ in (i), let $\cR'$ be the digraph obtained from $\cR$ by applying the following operations. If $e\notin\cR$, then add $(v,w)$ and, if $e\in\cR$, then subdivide $e$ with $v$. Similarly, if $e'\notin\cR$, then add $(v',w')$ and, if $e'\in\cR$, then subdivide $e'$ with $v'$.  It is now straightforward to check that $\cR'$ is an extension of $\cD$ in $\cN'$ with $|E_{\cN}-E_{\cR}|+1=|E_{\cN'}-E_{\cR'}|$.

We now turn to (iii). Let $p_u$ be the parent of $u$, and let $c_u$ be the child of~$u$ that is not $v$. Since $u$ is a tree vertex, $p_u$ and $c_u$ are well defined. Furthermore, let $\cM$ be the embedding of $\cD$ in $\cN$ that underlies $\cR$. By Lemma~\ref{l:root-ext-basic}(i), \mbox{$(u,c_u)\in\cR$}. \mbox{If $(p_u,u)\notin\cR$}, then the degree constraints of the vertices in $\cD$ imply that \mbox{$(u,c_u)\notin\cM$}. On the other hand, if $(p_u,u)\in\cR$, then $(u,c_u)\in\cR$ and, in turn, either both of $(p_u,u)$ and $(u,c_u)$ are in $\cM$ or neither.
Now obtain $\cM'$ from $\cM$ by replacing $(p_u,u)$ and $(u,c_u)$ with $(p_u,c_u)$ if  $(p_u,u)$ and $(u,c_u)$ are in $\cM$, and replacing $f$ with $(p_{u'},u')$ and $(u',c_{u'})$ if  $f\in\cM$.
Thus, as $\cM$ is the embedding of $\cD$ in~$\cN$ that underlies $\cR$, it  follows from the construction of $\cN'$ from $\cN$ that $\cM'$ is also an embedding of $\cD$ in $\cN'$. Hence, $\cD$ is a tree-child digraph of $\cN'$. To complete the proof, let $\cR'$ be the digraph obtained from $\cR$ by applying the following operations. If $u$ has in-degree zero in $\cR$, then delete $u$ and, if $u$ has in-degree one in $\cR$, then suppress $u$. Moreover, if $f\notin \cR$, then add $(u',c_{u'})$ and, if $f\in \cR$, then subdivide $(p_{u'},c_{u'})$ with $u'$. Again, by Lemma~\ref{l:root-ext-basic}, $u$ has either in-degree zero and out-degree one, or in-degree one and out-degree one in $\cR$ and, so, $\cR'$ is well defined. As $\cR$ is an extension of $\cD$ in $\cN$ with $e\notin\cR$ and $(u',v)\notin\cR'$, it now follows that $\cR'$ is an extension of $\cD$ in $\cN'$ with $|E_{\cN}-E_{\cR}|=|E_{\cN'}-E_{\cR'}|$.
\end{proof}

The next lemma and corollary show that there always exists a  tree-child $\SNPR$ sequence connecting two tree-child networks with certain desirable properties.\ These properties will be leveraged later to establish one of the two inequalities of Theorem~\ref{t:main}. 

\begin{lemma}\label{l:swap}
Let $(\cN_0,\cN_1,\cN_2,\ldots,\cN_t)$ be a tree-child $\SNPR$ sequence connecting two tree-child networks $\cN_0$ and $\cN_t$ on $X$. If there exists an $i\in\{0,1,2,\ldots,t-2\}$ such that $\cN_{i+1}$ is obtained from $\cN_i$ by an $\SNPR^+$ (resp.\ $\SNPR^\pm$) and $\cN_{i+2}$ is obtained from $\cN_{i+1}$ by an $\SNPR^-$, then one of the following holds:
\begin{enumerate}[{\rm (i)}]
\item $(\cN_0,\cN_1,\cN_2,\ldots,\cN_i,\cN_{i+3},\cN_{i+4},\ldots,\cN_{t-1},\cN_t)$ is a tree-child $\SNPR$ sequence connecting $\cN_0$ and $\cN_t$ of length $t-2$, 
\item $(\cN_0,\cN_1,\cN_2,\ldots,\cN_i,\cN_{i+2},\cN_{i+3},\ldots,\cN_{t-1},\cN_t)$ is a tree-child $\SNPR$ sequence connecting $\cN_0$ and $\cN_t$ of length $t-1$ such that $\cN_{i+2}$ is obtained from $\cN_{i}$ by an $\SNPR^\pm$, 
\item $(\cN_0,\cN_1,\cN_2,\ldots,\cN_i,\cN_{i+2},\cN_{i+3},\ldots,\cN_{t-1},\cN_t)$ is a tree-child $\SNPR$ sequence connecting $\cN_0$ and $\cN_t$ of length $t-1$ such that $\cN_{i+2}$ is obtained from $\cN_{i}$ by an $\SNPR^-$, or
\item there exists a tree-child $\SNPR$ sequence  $$(\cN_0,\cN_1,\cN_2,\ldots,\cN_i,\cN_{i+1}',\cN_{i+2},\cN_{i+3},\ldots,\cN_{t-1},\cN_t)$$ connecting $\cN_0$ and $\cN_t$ of length $t$ such that $\cN_{i+1}'$ is obtained from $\cN_i$ by an $\SNPR^-$ and $\cN_{i+2}$ is obtained from $\cN_{i+1}'$  by an $\SNPR^+$ (resp.\ $\SNPR^\pm$).
\end{enumerate}
\end{lemma}

\begin{proof}
Suppose that there exists an element $i\in\{0,1,2,\ldots, t-2\}$ such that $\cN_{i+1}$ is obtained from $\cN_i$ by an SNPR$^+$  and $\cN_{i+2}$ is obtained from $\cN_{i+1}$ by an SNPR$^-$. Let $e=(u,v)$ be the reticulation edge of $\cN_{i+1}$ that is added in obtaining $\cN_{i+1}$ from~$\cN_i$. Furthermore, let $e'=(u',v')$ be the reticulation edge in $\cN_{i+1}$ that is deleted in obtaining $\cN_{i+2}$ from $\cN_{i+1}$. If $e=e'$, then $\cN_i\cong\cN_{i+2}$, and so (i) holds. We may therefore assume that $e\ne e'$. If $v=v'$, let $w$ be the child of $v$ in $\cN_{i+1}$, and let $c_u$ and $p_u$ be the child and parent, respectively, of $u$ in $\cN_{i+1}$ such that $c_u\ne v$. Since $u$ is a tree vertex, $c_u$ and $p_u$ are well defined. Observe that $(u',w)$ and $(p_u,c_u)$ are tree edges in $\cN_i$. It now follows that $\cN_{i+2}$ can be obtained from $\cN_i$ by the SNPR$^\pm$ that deletes $(u',w)$, suppresses $u'$, subdivides $(p_u,c_u)$ with a new vertex $u$, and adds the edge $(u,w)$, and so (ii) holds. So assume that  $e\ne e'$ and $v\ne v'$. As $\cN_{i+1}$ is tree-child, it follows that $e'$  is not incident with $u$ or $v$. Thus, $e'$ is a reticulation edge of $\cN_i$. Let $\cN_{i+1}'$ be the phylogenetic network obtained from~$\cN_{i}$ by deleting $e'$ and suppressing $u'$ and $v'$. By Lemma~\ref{l:reticulation-edge}, $\cN_{i+1}'$ is tree-child. It is now straightforward to check that $\cN_{i+2}$ can be obtained from $\cN_{i+1}'$ by an SNPR$^+$, and so (iv) holds. 

\begin{figure}[t]
\centering
\scalebox{0.87}{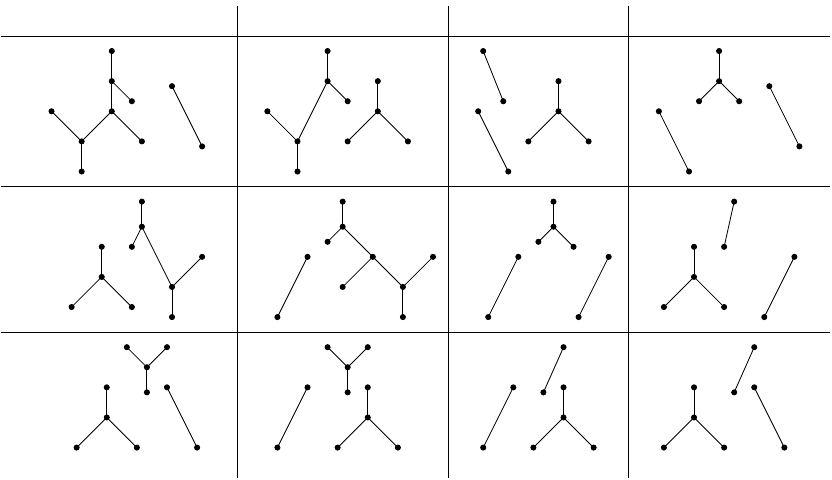}
\caption{The three cases in the last case analysis in the proof of Lemma~\ref{l:swap}, where  $\cN_{i+1}$ is obtained from $\cN_i$ by an SNPR$^\pm$, $\cN_{i+2}$ is obtained from $\cN_{i+1}$ by an SNPR$^-$, and $p_{u'}\ne p_u$ or $c_{u'}\ne c_u$. The three cases are (I) $e'=(p_u,c_u)$, (II)  $e'=(u',c_{u'})$, and (III) $e'\ne (p_u,c_u)$ and $e'\ne (u',c_{u'})$. Vertex labels $y_1$, $y_2$, and $y_3$ are only used to clarify the figure.}
\label{fig:lemma-6-2}
\end{figure}

Next suppose that there exists an element $i\in\{0,1,2,\ldots, t-2\}$ such that $\cN_{i+1}$ is obtained from $\cN_i$ by an SNPR$^\pm$  and $\cN_{i+2}$ is obtained from $\cN_{i+1}$ by an SNPR$^-$. Let $e=(u,v)$ be the edge of $\cN_i$ that is deleted in the process of obtaining $\cN_{i+1}$. Furthermore, let $p_u$ be the parent of $u$ and let $c_u$ be the child of $u$ that is not $v$ in $\cN_i$. Since $u$ is a tree vertex, $p_u$ and $c_u$ are well defined. Let $f=(p_{u'},c_{u'})$ be the edge of the digraph resulting from $\cN_i$ by deleting $e$ and suppressing $u$ that is subdivided with a new vertex $u'$ such that $(u',v)$ is an edge in $\cN_{i+1}$. If $p_{u'}=p_u$ and $c_{u'}=c_u$, then $\cN_i\cong\cN_{i+1}$, and so (iii) holds. Hence, we may assume that $p_{u'}\ne p_u$ or $c_{u'}\ne c_u$, and so $(p_u,c_u)$ is an edge in $\cN_{i+1}$.  Let $e'$ be the edge that is deleted in obtaining $\cN_{i+2}$ from $\cN_{i+1}$. There are three cases to consider, which are illustrated in Figure~\ref{fig:lemma-6-2}. First assume that $e'=(p_u,c_u)$. Then $(u,c_u)$ is a reticulation edge in~$\cN_i$. Let $w$ be the parent of $c_u$ that is not $u$ in $\cN_i$. Obtain $\cN_{i+1}'$ by deleting $(u,c_u)$ and suppressing the two resulting degree-two vertices. By Lemma~\ref{l:reticulation-edge}, $\cN_{i+1}'$ is tree-child. Moreover, noting that $(p_u,v)$ is an edge in $\cN_{i+1}'$, it follows that $\cN_{i+2}$ can be obtained from $\cN_{i+1}'$ by an SNPR$^\pm$ that deletes $(p_u,v)$, suppresses $p_u$, subdivides $(p_{u'},c_{u'})$ if $c_u\ne p_{u'}$ (resp.\ subdivides $(w,c_{u'})$ if $c_u= p_{u'}$) with a new vertex $u'$, and adds the edge $(u',v)$, and so (iv) holds. Second assume that $e'=(u',c_{u'})$. Then~$f$ is a reticulation edge in $\cN_i$. Noting that $p_{u'}$ is a tree vertex in $\cN_i$, let $s$ be the parent of $p_{u'}$, and let $t$ be the child of $p_{u'}$ that is not $c_{u'}$. Now obtain $\cN_{i+1}'$ from $\cN_i$ by deleting $f$ and suppressing the two resulting degree-two vertices, one of which is $p_{u'}$. Again by Lemma~\ref{l:reticulation-edge}, $\cN_{i+1}'$ is tree-child. Furthermore, $(s,t)$ is an edge in~$\cN_{i+1}'$. Then obtain $\cN_{i+2}$ from $\cN_{i+1}'$ by an SNPR$^\pm$ that deletes $(u,v)$, suppresses $u$, subdivides $(s,t)$ with a new vertex $p_{u'}$, and adds the edge $(p_{u'},v)$. Again (iv) holds. Third assume that $e'\ne (p_u,c_u)$ and $e'\ne (u',c_{u'})$. Then $e'$ is  an edge of $\cN_i$. Let $\cN'_{i+1}$ be the tree-child network obtained from $\cN_i$ by deleting $e'$ and suppressing the two resulting degree-two vertices. Since $e$ and $f$ are edges of  $\cN_{i+1}'$, it now follows that $\cN_{i+2}$ can be obtained from $\cN_{i+1}'$ by an SNPR$^\pm$ that deletes $e$, suppresses $u$, subdivides $f$ with a new vertex $u'$, and adds the edge $(u',v)$. Again, (iv) holds, thereby completing the proof of the lemma.
\end{proof}

\begin{corollary}\label{c:nice-seq}
Let $\cN$ and $\cN'$ be two tree-child networks on $X$. Then there exists a tree-child $\SNPR$ sequence $(\cN=\cN_0,\cN_1,\cN_2,\ldots,\cN_t=\cN')$ that connects $\cN$ and~$\cN'$ such that either $\cN_t$ is obtained from $\cN_{t-1}$ by an $\SNPR^+$ or $\SNPR^\pm$, or $\cN_{i}$ is obtained from $\cN_{i-1}$ by an $\SNPR^-$  for each $i\in\{1,2,\ldots, t\}$.
\end{corollary}

\begin{proof}
The corollary follows from repeated applications of Lemma~\ref{l:swap}.
\end{proof}

The proof of Theorem~\ref{t:main} is an amalgamation of the next two lemmas. 

\begin{lemma}\label{l:dtc<mtc}
Let $\cN$ and $\cN'$ be two tree-child networks on $X$. Then $$d_{\tc}(\cN,\cN’) \leq m_{\tc}(\cN,\cN’).$$
\end{lemma}

\begin{proof}
Let $\cD=\{D_\rho,D_1,D_2,\ldots,D_k\}$ be an agreement tree-child digraph for $\cN$ and $\cN'$, and let $\cR=\{R_\rho,R_1,R_2,\ldots,R_k\}$ and $\cR'=\{R_\rho',R_1',R_2',\dots, R_k'\}$ be an extension of $\cD$ in $\cN$ and $\cN'$, respectively.  Furthermore, let $c_\cD=|E_\cN-E_\cR|$ and $c_\cD'=|E_{\cN'}-E_{\cR'}|$ be the cut size of  $\cD$ in $\cN$ and $\cN'$, respectively. By Lemma~\ref{l:extension-equiv}, we may assume that $\cR$ and $\cR'$ is a root extension of $\cD$ in $\cN$ and $\cN'$, respectively. We show by induction on $c_\cD+c_\cD'$ that $d_{\tc}(\cN,\cN’) \leq c_\cD+c_\cD'$. The lemma then follows by choosing $\cD$ to be an agreement tree-child digraph for $\cN$ and $\cN'$ such that $c_\cD+c_\cD'=m_\tc(\cN,\cN')$. 

If $c_\cD+c_\cD'=0$, then $\cN\cong\cN'$ and consequently $d_\tc(\cN,\cN')=0$. Hence, the result follows. Now assume that $c_\cD+c_\cD'\geq 1$ and that the result holds for all pairs of tree-child networks $\cN_1$ and $\cN_1'$ on the same leaf set that have an agreement tree-child digraph $\cD_1$ with cut size  $c_{\cD_1}$ and $c_{\cD_1}'$  of $\cD_1$ in $\cN_1$ and $\cN_1'$, respectively, such that $c_{\cD_1}+c_{\cD_1}'<c_\cD+c_\cD'$. 

We first establish the lemma for a case that is easy to deal with. To this end, assume that there exists a reticulation edge $e=(u,v)$ in $\cN$ or $\cN'$ that is not contained in $\cR$ or $\cR'$, respectively. Without loss of generality, we may assume that $e\in\cN$. As $\cN$ is tree-child, $u$ is a tree vertex. Let $\cN''$ be the phylogenetic network obtained from $\cN$ by an SNPR$^-$ that deletes $e$ and suppresses the two resulting degree-two vertices. By Lemma~\ref{l:reticulation-edge}, $\cN''$ is tree-child. Furthermore, by Lemma~\ref{l:mod-ext}(i), $\cD$ is a tree-child digraph for $\cN''$ and there exists a root extension $\cR''$ of $\cD$ in $\cN''$ such that  $c_\cD-1= c_\cD''$, where $c_\cD''=|E_{\cN''}-E_{\cR''}|.$ 
Since $c_\cD''+c_\cD'<c_\cD+c_\cD'$, it now follows from the induction assumption that $$d_\tc(\cN'',\cN')\leq c_\cD''+c_\cD'.$$ Hence, there exists a tree-child SNPR sequence $\sigma$ connecting $\cN''$ and $\cN'$ with $w(\sigma) \leq c_\cD''+c_\cD'$. 
Moreover, since $\cN''$ can be obtained from $\cN$ by a single SNPR$^-$, we have
$$d_\tc(\cN,\cN')\leq  1+c_\cD''+c_\cD'= c_\cD+c_\cD',$$
thereby establishing the lemma under the assumption that such an edge $e$ exists.

To complete the proof, we may now assume that \\

\noindent {\bf (A)} each reticulation edge of $\cN$ and $\cN'$ is contained in $\cR$ and $\cR'$, respectively. \\

\noindent Hence, by symmetry, there exists a tree edge $e=(u,w)$ in $\cN$ that is not in $\cR$. Choose $e$ such that each directed path from $w$ to a leaf in $\cN$ only consists of edges in $\cR$. Since $\cN$ is acyclic such an edge exists.  Let $D_i$ be the element in $\cD$ with $i\in\{\rho,1,2,\ldots,k\}$ such that the root extension $R_i$ of $D_i$ in $\cN$ contains $w$. By 
the choice of $e$, $R_i$ exists and each edge in $\cN$ that lies on a directed path from $w$ to a leaf is an element of $R_i$. Thus, if $u=\rho$, then $\cN\cong\cN'$ and the result follows as $d_\tc(\cN,\cN')=0$. We may therefore assume that $u\ne \rho$. The next statement follows from Lemma~\ref{l:root-ext-basic}(i), the additional assumption that $u\ne\rho$, and assumption (A).

\begin{sublemma}\label{one}
In $\cN$, the vertex $u$ is a tree vertex, and $w$ is either a tree vertex or a leaf. 
\end{sublemma}

\noindent By the choice of $e$, observe that the root path of $w$ consists only of $w$. In turn, $w$ has in-degree zero and out-degree zero or two in $R_i$. Hence, $w$ corresponds to a unique vertex, say $w_\cD$, of $D_i$. Let $r_{w'}$ be the vertex in $\cN'$ that $w_\cD$ corresponds to.  Now consider the root extension $R_i'$  of $D_i$ in $\cN'$.  Let $w'$ be the first vertex of the root path of $r_{w'}$. In contrast to the root path of $w$ in $R_i$, observe that the root path of $r_{w'}$ may consist of more than a single vertex in which case there is a directed path of length at least one from $w'$ to $r_{w'}$. Since $\rho\in\cD$ and $w_\cD\ne\rho$, it follows that $w'\ne\rho$. Hence, by (A), we have that $w'$ is a leaf, or has in-degree one and out-degree two in $\cN'$. Let $v'$ be the parent of $w'$ in $\cN'$. As $(v',w')\notin \cR'$, the next statement follows from Lemma~\ref{l:root-ext-basic}(i).

\begin{sublemma}\label{three}
In $\cN'$, the edge $e'=(v',w')$ is a tree edge, and $v'$ is either $\rho$ or a tree vertex in $\cN'$. 
\end{sublemma}

Let $D_j$ be the element in $\cD$ with $j\in\{\rho,1,2,\ldots,k\}$ such that the root extension $R_j'$ of $D_j$ in $\cN'$ contains $v'$. By  Lemma~\ref{l:root-ext-basic}(ii), $R_j'$ exists. We may have $i=j$. We next construct a  digraph $\cD'$ from $\cD$ and a  network $\cN''$ from $\cN$. After detailing the construction, we show that $\cN''$ is a tree-child network that  can be obtained from~$\cN$ by a  single SNPR$^\pm$, and that $\cD'$ is an agreement tree-child digraph for $\cN''$ and $\cN'$. Guided by the second part of~\eqref{three} and noting that, if $v'=\rho$, then $\rho$ is a singleton component in $\cD$, there are three cases to consider, which are illustrated in Figure~\ref{fig:lemma-6-4}:

\begin{figure}[t]
\centering
\scalebox{0.87}{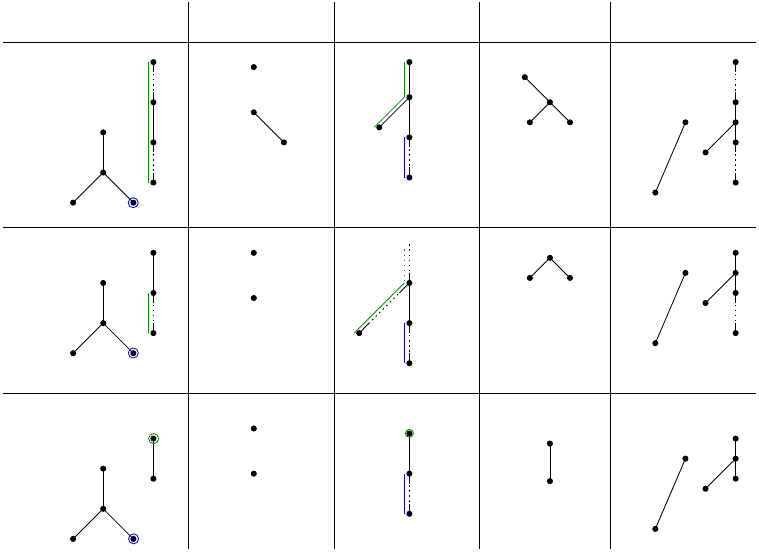}
\caption{The three Cases (C1)--(C3) as described in the proof of Lemma~\ref{l:dtc<mtc}. Blue in $\cN$ and $\cN'$ indicates edges and vertices in $R_i$ and $R_i'$, and green in $\cN$ and $\cN'$ indicates edges and vertices in $R_j$ and $R_j'$. Vertex labels $y_1$ and $y_2$ are only used to clarify the figure.}
\label{fig:lemma-6-4}
\end{figure}

\begin{enumerate}[(C1)]
\item Suppose that $v'$ is not a vertex of a root path of a vertex in $\cR'$. Clearly $v'\ne \rho$. Then $v'$ has in-degree one and out-degree one in $R_j'$. Recall that each edge in $D_j$ corresponds to a unique directed path in $R_j'$ that connects the two end vertices of that edge. Let $(s_\cD,t_\cD)$ be the edge in $D_j$ that corresponds to a directed path in $R_j'$ that contains the two edges incident with $v'$. Obtain $D'_{ij}$ from $D_i$ and $D_j$ by subdividing $(s_\cD,t_\cD)$ with a new vertex $p_\cD$ and adding the edge $(p_\cD,w_\cD)$. Turning to $\cN$, let $f$ be an edge that lies on the  directed path in $R_j$ that corresponds to the edge $(s_\cD,t_\cD)$ in $D_j$. Obtain $\cN''$ from $\cN$ by deleting $e$, suppressing $u$, subdividing $f$ with a new vertex $p$, and adding the edge $(p,w)$.
\item Suppose that $v'\ne\rho$ and that $v'$ is a vertex of a root path $P$ of $\cR'$. Let $r_{v'}$ be the last vertex of $P$ in $\cN'$. Noting that $r_{v'}$ corresponds to a vertex $v_\cD$ of in-degree zero in $D_j$, obtain $D'_{ij}$ from $D_i$ and $D_j$ by adding the two edges $(p_\cD,v_\cD)$ and $(p_\cD, w_\cD)$, where $p_\cD$ is a new vertex. Turning to $\cN$, let $r_v$ be the vertex in $R_j$ that $v_\cD$ corresponds to, and let $v$ be the first vertex of the root path of $r_v$. As $v'\ne \rho$ and $\cD$ is an agreement digraph for $\cN$ and $\cN'$, we have $v\ne\rho$. Moreover, it follows from (A) that $v$ is not a reticulation in~$\cN$. Thus,~$v$ has a unique parent, say $t$, in $\cN$. Then, obtain $\cN''$ from $\cN$ by deleting $e$, suppressing $u$, subdividing the edge $f=(t,v)$ with a new vertex~$p$, and adding the edge $(p,w)$.
\item Suppose that $v'=\rho$. As $\rho$ has out-degree one in $\cN$ and $\cN'$, each of $D_j$ and~$R_j'$ consist of the isolated vertex $\rho$ only. Then obtain $D'_{ij}$ from $D_i$ and~$D_j$ by adding a new edge $(\rho,w_\cD)$. Moreover, obtain $\cN''$ from $\cN$ by deleting $e$, suppressing $u$, subdividing the edge $f$ that is directed out of $\rho$ with a new vertex $p$, and adding the edge $(p,w)$.
\end{enumerate}

\noindent As $\cN'$ does not contain a directed cycle, it follows from the construction that $D'_{ij}$ is acyclic in all three cases. Hence, as $\cD$ is a phylogenetic digraph for $\cN'$, $$\cD'=(\cD-\{D_i,D_j\})\cup\{D'_{ij}\}$$ is a phylogenetic digraph of $\cN'$. Let $E_i'$ and $E_j'$ be the edge set of $R_i'$ and $R_j'$ respectively, and let $R_{ij}'$ be the subgraph of $\cN'$ induced by the edge set $E_i'\cup E_j'\cup\{\{v',w'\}\}$. Since $\cR'$ is a root extension of $\cD$ in $\cN'$, it again follows from the construction that $$(\cR'-\{R_i',R_j'\})\cup\{R_{ij}'\}$$ is a root extension of $\cD'$ in $\cN$. 

We next turn to $\cD'$ and show that $\cD'$ is tree-child. Since $\cD$ is tree-child, it follows from the definition that $\cR'$ is tree-child. Moreover, since $w'$ has in-degree zero in $R_i'$, it follows that $w'$ has in-degree one in $R_{ij}'$. It is now straightforward to check that $(\cR'-\{R_i',R_j'\})\cup\{R_{ij}'\}$  is tree-child. Hence, again by definition, $\cD'$ is also tree-child. 

The following statement is now an immediate consequence of the construction.

\begin{sublemma}\label{two}
The cut size of $\cD'$ in $\cN'$ is $c_\cD'-1$.
\end{sublemma}

Next, we establish that $\cN''$ is a tree-child network on $X$.

\begin{sublemma}\label{four}
The network $\cN''$ is acyclic.
\end{sublemma}

\begin{proof}
Using the same notation as in the construction of $\cN''$ from $\cN$, recall that $e=(u,w)$ is the edge in $\cN$ that is deleted and that $f$ is the edge in $\cN$ that is subdivided with $p$ in the process of obtaining $\cN''$. To ease reading, let $f=(p_p,c_p)$ regardless of which of (C1)--(C3) applies. Since $\cN$ is acyclic, any directed cycle in~$\cN''$ contains $p$. If $\cN''$ has been obtained from $\cN$ as described in (C3), then $\cN''$ is acyclic because $p$ has in-degree one and out-degree two in $\cN''$ and is adjacent to~$\rho$. Hence, we may assume that $\cN''$ has been obtained as described in (C1) or~(C2). Towards a contradiction, assume that $\cN''$ contains a directed cycle. Then there exists a directed path $P$ from $w$ to $c_p$ in $\cN$ whose last edge is $f$. If $\cN''$ has been obtained from $\cN$ as described in  (C2), then $f\in\cN$ and $f\notin\cR$, a contradiction  because $P$ contains an edge not in $\cR$, which is not possible by the choice of $e$ as described in the paragraph following the statement of assumption (A). On the other hand, if $\cN''$ has been obtained from $\cN$ as described in  (C1), then, again by the choice of $e$ and the existence of $P$, we have $R_i=R_j$. Since $\cD$ is an agreement digraph for $\cN$ and $\cN'$, it follows that the edge $(s_\cD,t_\cD)$ in $D_i$ can be reached from~$w_\cD$, thereby contradicting that $D'_{ij}$ is acyclic. 
\end{proof}

\noindent It now follows from the construction of $\cN''$ from $\cN$ and~\eqref{four} that $\cN''$ is a phylogenetic network on $X$. For the remainder of the proof, let $D_u$ be the element in~$\cD$ with $u\in\{\rho,1,2,\ldots,k\}$ such that the root extension $R_u$ of $D_u$ in $\cN$ contains $u$. By Lemma~\ref{l:root-ext-basic}(ii), $R_u$ exists.

\begin{sublemma}\label{five}
The phylogenetic network $\cN''$ is tree-child.
\end{sublemma}

\begin{proof}
Again using the same notation as in the construction of $\cN''$ from $\cN$, it follows from~\eqref{one} that the newly added edge $(p,w)$ in $\cN''$ is a tree edge. Noting that $u$ is a tree vertex by~\eqref{one} in $\cN$, let $p_u$ be the parent of $u$, and let $c_u$ be the child of $u$ that is not $w$ in $\cN$. Observe that $(p_u,c_u)$ is an edge in $\cN''$. Now assume that $\cN''$ is not tree-child. 

First suppose that $\cN''$ contains a pair of parallel edges. Then $(p_u,c_u)$, $(p_u,u)$, and $(u, c_u)$ are edges of an underlying three-cycle in $\cN$.\ Assumption (A) and Lemma~\ref{l:root-ext-basic}(i) imply that all three edges incident with $c_u$ are edges in $R_u$. If $(p_u,u)\notin R_u$, then $u$ has in-degree zero and out-degree one in $R_u$. It follows that~$u$ is a vertex of $\cR$ but not a vertex of the embedding of $\cD$ in $\cN$ that underlies $\cR$. Hence, the unique child of $u$ in $R_u$ has in-degree one in $R_u$ because $\cR$ is a root extension of $\cD$, a contradiction as $c_u$ has in-degree two in $R_u$. Thus, $(p_u,u)\in R_u$. It follow that $D_u$ contains a pair of parallel edges because $e\notin \cR$, a contradiction to $\cD$ being tree-child. 

Second suppose that $\cN''$ contains an edge that is incident with two reticulations. Then $p_u$ and $c_u$ are reticulations in $\cN$. It follows from (A) and Lemma~\ref{l:root-ext-basic}(i), that~$R_u$ contains  the three edges incident with $c_u$ and the three edges incident with $p_u$. Thus, $D_u$ contains an edge that is incident with two reticulations because~$e\notin\cR$, another contradiction. 

Third suppose that $\cN''$ contains a pair of sibling reticulations. Then $c_u$ is a reticulation and $p_u$ is a tree vertex whose child that is not $u$, say $s_u$, is a reticulation in $\cN$.  Again by (A) and Lemma~\ref{l:root-ext-basic}(i), $R_u$ contains all three edges that are incident with $c_u$ and there exists an element $R_{u'}\in\cR$ with $u'\in\{\rho,1,2,\ldots,k\}$ such that~$R_{u'}$ contains all three edges incident with $s_u$. If $u\ne u'$, then $(p_u,u)\notin\cR$ and, thus, $u$ has in-degree zero and out-degree one in $R_u$. It follows that $u$ is a vertex of $\cR$ but not a vertex of the embedding of $\cD$ in $\cN$ that underlies $\cR$. Hence, the unique child of $u$ in $R_u$ has in-degree one in $R_u$, a contradiction as $c_u$ has in-degree two in~$R_u$. We may therefore assume that $u=u'$. But then $R_u$ contains a pair of sibling reticulations $s_u$ and $c_u$ because $e\notin\cR$, a final contradiction. 
\end{proof}

\noindent It now follows from~\eqref{four} and \eqref{five} and the construction as detailed in \mbox{(C1)--(C3)} that $\cN''$ is a tree-child network on $X$ that can be obtained from $\cN$ by a single SNPR$^\pm$. We next show that $\cD'$ is a phylogenetic digraph for $\cN''$. To this end, we construct a root extension of $\cD'$ in $\cN''$. 

If $\cN''$ has been obtained from $\cN$ as described in (C1), obtain a root extension $R_{ij}$ of $D_{ij}'$ from $R_i$ and $R_j$ by subdividing $f$ in $R_j$ with a new vertex $p$  and  adding the edge $(p,w)$. Otherwise, if $\cN''$ has been obtained from $\cN$ as described in (C2) or (C3), obtain $R_{ij}$ from $R_i$ and $R_j$ by adding the edge $(p,v)$, where $p$ is a new vertex, and adding the edges $(p,v)$ and $(p,w)$. Then, as $\cR$ is a root extension of $\cD$ in $\cN$ and $u$ is a tree vertex in $\cN$ by~\eqref{one}, it follows that the digraph obtained from  $$\cR''=(\cR-\{R_i,R_j\})\cup \{R_{ij}\}$$ by suppressing (resp.\ deleting) $u$ if it has in-degree one (resp.\ zero) in $\cR''$ is a root extension of $\cD'$ in $\cN''$. Thus, $\cD'$ is an agreement tree-child digraph for $\cN''$ and $\cN'$. 

The next statement is again an immediate consequence of the construction of~$\cR''$.

\begin{sublemma}\label{six}
The cut size of $\cD'$ in $\cN''$ is $c_\cD-1$.
\end{sublemma}

\noindent  By combining~\eqref{two} and~\eqref{six}, it now follows from the induction assumption that 
$$d_\tc(\cN'',\cN')\leq c_\cD-1+c_\cD'-1.$$
Hence, there exists a tree-child SNPR sequence $\sigma$ connecting $\cN''$ and $\cN'$ with $w(\sigma)\leq c_\cD-1+c_\cD'-1$. 
Since $\cN''$ can be obtained from $\cN$ by a single SNPR$^\pm$, we have
$$d_\tc(\cN,\cN')\leq 2+c_\cD-1+c_\cD'-1= c_\cD+c_\cD' .$$
The lemma now follows.
\end{proof}

\begin{figure}[t]
\center
\scalebox{0.87}{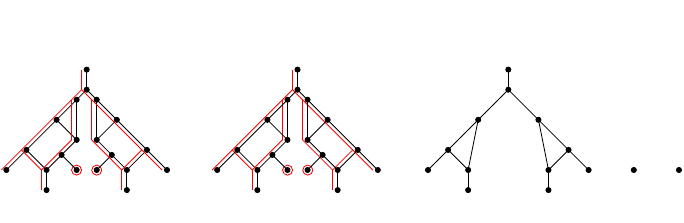}
\caption{An example of two tree-child networks $\cN$ and $\cN'$ and an agreement tree-child digraph $\cD$ for $\cN$ and $\cN'$ for which \mbox{$6=d_\tc(\cN,\cN')$} but $m_\tc(\cN,\cN')=8$. An extension of $\cD$ in $\cN$ and $\cN'$ is indicated by red lines.}
\label{fig:counterexample}
\end{figure}

Figure~\ref{fig:counterexample} shows  two tree-child networks for which the inequality established in Lemma~\ref{l:dtc<mtc} is strict. However, the next lemma shows that, for two tree-child networks $\cN$ and $\cN'$, the difference $m_{\rm tc}(\cN, \cN')-d_{\rm tc}(\cN, \cN')$ cannot be arbitrary large. In preparation for the  lemma, we need an additional definition. Let $\cD$ be a phylogenetic digraph of a phylogenetic network $\cN$ on $X$. Furthermore, let $\cR$ be an extension of $\cD$ in $\cN$, and let $\cM$ be the embedding that underlies $\cR$. Now consider a directed path $P$ in $\cM$. Let $V=\{v_1,v_2,\ldots,v_n\}$ be the subset of reticulations in~$\cN$ that lie on $P$. Then the {\it path extension} of $P$ contains precisely all edges of $P$ and, additionally, each edge of a maximal length directed path in $\cN$ that only consists of edges in $E_\cR-E_\cM$ and ends at a vertex in $V$. Note that the path extension of~$P$ may contain each edge of~$P$ and no additional edge, even if $V\ne\emptyset$.

\begin{lemma}\label{l:mtc<dtc}
Let $\cN$ and $\cN'$ be two tree-child networks on $X$. Then $$\frac 1 2 m_{\tc}(\cN,\cN’) \leq d_{\tc}(\cN,\cN’).$$
\end{lemma}

\begin{proof}
Let $\sigma=(\cN=\cN_0,\cN_1,\cN_2,\ldots,\cN_t=\cN')$ be a tree-child SNPR sequence connecting $\cN$ and $\cN'$ such that $\cN_{i-1}$ and $\cN_{i}$ are non-isomorphic for each $i\in\{1,2,\ldots,t\}$. It follows from Bordewich et al.~\cite[Proposition 3.2]{bordewich17} that $\sigma$ exists. 
By Corollary~\ref{c:nice-seq}, we may assume that $\cN_t$ can be obtained from $\cN_{t-1}$ by an SNPR$^+$ or an SNPR$^\pm$, or $\cN_i$ can be obtained from $\cN_{i-1}$ by an SNPR$^-$ for each $i\in\{1,2,\ldots,t\}$. If the latter holds, then, by the reversibility of SNPR, $$(\cN'=\cN_t,\ldots,\cN_2,\cN_1,\cN_0=\cN)$$ is a tree-child SNPR sequence connecting $\cN$ and $\cN'$ and  $\cN_i$ is obtained from $\cN_{i+1}$ by an SNPR$^+$ for each $i\in\{t-1,t-2,\ldots,0\}$. Hence, we may assume without loss of generality that $\cN_t$ can be obtained from $\cN_{t-1}$ by either an SNPR$^+$ or an SNPR$^\pm$.
We show by induction on  $t$ that there exist an agreement tree-child digraph $\cD$ for $\cN$ and $\cN'$ and extensions $\cR$ and $\cR'$ of $\cD$ in $\cN$ and $\cN'$, respectively, such that $\frac 1 2 m_\tc(\cN,\cN’) \leq w(\sigma)$. The lemma then follows by choosing $\sigma$ such that $w(\sigma)=d_\tc(\cN,\cN')$. 

If $t=1$, then there are two cases to consider. First assume that $\cN'$ can be obtained from $\cN$ by an SNPR$^+$. Then $w(\sigma)=1$, and $\cN$ is an agreement tree-child digraph for $\cN$ and $\cN'$. Trivially, there is an extension $\cR$ of $\cN$ in $\cN$ and an extension $\cR'$ of $\cN$ in $\cN'$ such that $|E_\cN-E_\cR|+|E_{\cN'}-E_{\cR'}|=1$ and, thus, $\frac 1 2 m_\tc(\cN,\cN')\leq \frac 1 2 \cdot 1<w(\sigma)$. 
Second assume that $\cN'$ can be obtained from $\cN$ by an SNPR$^\pm$ in which case $w(\sigma)=2$. Recalling the global assumption stated at the end of Section~\ref{sec:measures}, let $e=(u,v)$ be the tree edge in $\cN$ that is deleted in the process of obtaining $\cN'$ from $\cN$. Let $p_u$ be the parent of $u$, and let $c_u$ be the child of $u$ in $\cN$ that is not $v$. Since $u$ is a tree vertex, $p_u$ and $c_u$ are well defined. If~$p_u$ is a tree vertex, let $s$ be the  child of $p_u$ in $\cN$ that is not $u$. Furthermore, if $c_u$ is a reticulation, let $s'$ be the parent of $c_u$ in $\cN$ that is not $u$. If $s$ and $c_u$ are both reticulations, let $\cD$ be the leaf-labelled acyclic digraph $\cD$ obtained from $\cN$ by deleting $e$ and $(s',c_u)$, and suppressing $u$, $c_u$, and $s'$. Otherwise, if at least one of $s$ and $c_u$ is not a reticulation, let $\cD$ be the leaf-labelled acyclic digraph $\cD$ obtained from $\cN$ by deleting $e$ and suppressing $u$. In both cases, $\cD$ is an agreement digraph of $\cN$ and $\cN'$. We next show that $\cD$ is tree-child. If $\cD$ contains a pair of parallel edges or a stack, then $\cN\cong\cN'$, a contradiction to the choice of $\sigma$. On the other hand, if $\cD$ contains a pair of sibling reticulations, then $s$ and $c_u$ are reticulations in~$\cN$. By construction, it follows that there is no reticulation in $\cD$ that corresponds to $c_u$. Hence, $\cD$ is tree-child. Moreover, there are  extensions $\cR$ of $\cD$ in $\cN$ and $\cR'$ of $\cD$ in $\cN'$ such that  $|E_\cN-E_\cR|+|E_{\cN'}-E_{\cR'}|\leq 2+2=4$, where the first inequality becomes an equality only if $s$ and $c_u$ are both reticulations in $\cN$. It now follows that  $\frac 1 2 m_\tc(\cN,\cN')\leq \frac 1 2\cdot 4=  w(\sigma)$. This completes the proof of the base case.

Now suppose that $t>1$ and that the lemma holds for all pairs of tree-child networks  for which there exists a tree-child SNPR sequence connecting the two networks of length less than $t$. Let $$\sigma_1=(\cN_0,\cN_1,\cN_2,\ldots,\cN_{t-1})\text{ and }\sigma_2=(\cN_{t-1},\cN_t).$$ By Corollary~\ref{c:nice-seq}, we may again assume that $\cN_{t-1}$ can be obtained from $\cN_{t-2}$ by an SNPR$^+$ or an SNPR$^\pm$, or $\cN_i$ can be obtained from $\cN_{i-1}$ by an SNPR$^-$ for each $i\in\{1,2,\ldots,t-1\}$.  Observe that $w(\sigma)=w(\sigma_1)+w(\sigma_2)$. By the induction assumption, we have 
$$\frac 1 2 m_\tc(\cN_0,\cN_{t-1})\leq w(\sigma_1).$$ 
Hence, there exist a maximum agreement tree-child digraph $\cD'$ for $\cN_0$ and $\cN_{t-1}$ and extensions $\cR_0'$ and $\cR_{t-1}'$ of $\cD'$ in $\cN_0$ and $\cN_{t-1}$, respectively, such that 
\begin{equation}\label{induction}
\frac 1 2 m_\tc(\cN_0,\cN_{t-1})= \frac 1 2(|E_{\cN_0}-E_{\cR'_0}| +|E_{\cN_{t-1}}-E_{\cR'_{t-1}}|)\leq w(\sigma_1).
\end{equation}
Let $\cM'_0$ (resp.\ $\cM'_{t-1}$) be the embedding of $\cD'$ in $\cN_0$ (resp.\ $\cN_{t-1}$) that underlies $\cR'_0$ (resp.\ $\cR'_{t-1}$).

Assume that $\cN_t$ can be obtained from $\cN_{t-1}$ by an SNPR$^+$, in which case \mbox{$w(\sigma_2)=1$.}  Let $(u,w)$ and $(u',w')$ be the two edges in $\cN_{t-1}$ that are subdivided with a new vertex $v$ and $v'$, respectively, in obtaining $\cN_t$. Since $\cN_t$ is tree-child, $(u,w)$ and $(u',w')$ are tree edges. Furthermore, either $(v,v')$ or $(v',v)$ is a reticulation edge in $\cN_t$. Without loss of generality, we may assume that $(v',v)$ is a reticulation edge in $\cN_t$. It now follows from Lemma~\ref{l:mod-ext}(ii) that $\cD'$ is also a tree-child digraph for $\cN_{t}$ and there exists an extension $\cR_t'$ of $\cD'$ in $\cN_t$ such that
$$|E_{\cN_{t-1}}-E_{\cR'_{t-1}}|+1=|E_{\cN_{t}}-E_{\cR'_{t}}|.$$ Hence, we have
\begin{eqnarray}
\frac 1 2 m_\tc(\cN,\cN')&\leq& \frac 1 2 (|E_{\cN_0}-E_{\cR'_0}|+ |E_{\cN_{t}}-E_{\cR'_{t}}|)\nonumber\\
&= & \frac 1 2 (|E_{\cN_0}-E_{\cR'_0}|+ |E_{\cN_{t-1}}-E_{\cR'_{t-1}}|+1)\nonumber\\
&<& w(\sigma_1)+w(\sigma_2)=w(\sigma),\nonumber
\end{eqnarray}
where the last inequality follows from Equation~\eqref{induction} and the fact that $w(\sigma_2)=1$.

For the remainder of the proof, we may therefore assume that $\cN_t$ is obtained from $\cN_{t-1}$ by an SNPR$^\pm$ in which case $w(\sigma_2)=2$.\ Let $e=(u,v)$ be the edge in~$\cN_{t-1}$ that is deleted in obtaining $\cN_t$ from $\cN_{t-1}$.\ By the definition of SNPR$^\pm$ and the global assumption, $u$ and $v$ are both tree vertices. Let $p_u$ be the parent of~$u$, and let $c_u$ be the child of $u$ with $c_u\ne v$ in $\cN_{t-1}$. Observe that $(p_u,c_u)$ is an edge in~$\cN_t$. Furthermore, let $(p_{u'},c_{u'})$ be the edge in $\cN_{t-1}$ that is subdivided with a new vertex $u'$ in obtaining $\cN_t$. Then $(u',v)$, $(p_{u'},u')$, and $(u',c_{u'})$ are edges in~$\cN_t$. 

First assume that $e\notin \cR'_{t-1}$. It  follows from Lemma~\ref{l:mod-ext}(iii) that $\cD'$ is also a tree-child digraph for $\cN_{t}$ and there exists an extension $\cR'_t$ of $\cD'$ in $\cN_t$ such that
$$|E_{\cN_{t-1}}-E_{\cR'_{t-1}}|=|E_{\cN_{t}}-E_{\cR'_{t}}|$$ 
and, therefore, again by Equation~\eqref{induction},
\begin{eqnarray}
\frac 1 2 m_\tc(\cN,\cN')&\leq&\frac 1 2 (|E_{\cN_0}-E_{\cR'_0}|+ |E_{\cN_{t}}-E_{\cR'_{t}}|)\nonumber\\
&=&\frac 1 2 (|E_{\cN_0}-E_{\cR'_0}|+ |E_{\cN_{t-1}}-E_{\cR'_{t-1}}|)\nonumber\\
&<& w(\sigma_1)+w(\sigma_2)=w(\sigma).\nonumber
\end{eqnarray}

Hence, we may assume that $e\in\cR'_{t-1}$. Let $R'_u$ be the element in $\cR'_{t-1}$ that contains  $e$,  let $R'_{c_u}$ be the element in $\cR'_{t-1}$ that contains $c_u$, and let $R'_{c_{u'}}$ be the element in $\cR'_{t-1}$ that contains $c_{u'}$. Recall that $R'_u$, $R'_{c_u}$, and $R'_{c_{u'}}$ are not necessarily pairwise distinct. If $(p_{u'},c_{u'})\in R'_{c_{u'}}$, then set $R_{c_{u'}}$ to be the directed graph obtained from  $R'_{c_{u'}}$ by subdividing $(p_{u'},c_{u'})$ with a new vertex $u'$. Otherwise, if \mbox{$(p_{u'},c_{u'})\notin R'_{c_{u'}}$,}  then set $R_{c_{u'}}$ to be the directed graph obtained from  $R'_{c_{u'}}$ by adding $(u',c_{u'})$. The construction is shown in Figure~\ref{fig:Rcu'}. Intuitively, $R_{c_{u'}}$ is an extension of a component of a phylogenetic  digraph in $\cN_t$. Lastly, if $R_u'=R'_{c_{u'}}$, then set $R'_u=R_{c_{u'}}$ and, if $R_{c_u}'=R'_{c_{u'}}$, then set $R'_{c_u}=R_{c_{u'}}$ to account for the modification in obtaining $R_{c_{u'}}$ from $R'_{c_{u'}}$.

\begin{figure}[t]
\center
\scalebox{0.87}{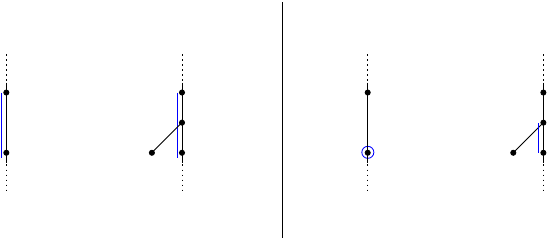}
\caption{The two cases in the construction of $R_{c_{u'}}$ from $R'_{c_{u'}}$. Blue indicates vertices and edges of $R_{c_{u'}}$ and $R'_{c_{u'}}$.}
\label{fig:Rcu'}
\end{figure}

Assume that $e\notin\cM'_{t-1}$. Then $(u,c_u)\notin\cR'_{t-1}$. It again follows from the construction of $\cN_t$ from $\cN_{t-1}$ that $\cD'$ is a phylogenetic digraph of $\cN_{t}$. Guided by~$\cR'_{t-1}$, we next construct an extension of $\cD'$ in $\cN_t$. Let $W$ be the subset of vertices of $\cN_{t-1}$ that lie on a directed path from a vertex with in-degree zero to $u$ in $R'_u$.  

\begin{enumerate}
\item [(R1)] If $u$ is the only element in $W$ and $R'_u\ne R'_{c_u}$, then obtain $R_u$ from $R'_u$ by deleting $u$, and set  $R_{c_u}=R'_{c_u}$. 
\item [(R2)] If $W$ contains $u$ and $|W|\geq 2$, and $R'_u\ne R'_{c_u}$, then obtain $R_u$ from $R'_u$ by deleting each vertex in $W$, and obtain $R_{c_u}$ from $R'_{c_u}$ by adding $(p_u,c_u)$ and each edge of $R'_u$ that joins two vertices in $W-\{u\}$. 
\item [(R3)] If $u$ is the only element in $W$ and $R'_u=R'_{c_u}$, then obtain $R_u$ from $R'_u$ by deleting $u$.
\item [(R4)] If $W$ contains $u$ and $|W|\geq 2$, and $R'_u=R'_{c_u}$, then obtain $R_u$ from $R'_u$ by deleting $u$ and adding $(p_u,c_u)$. 
\end{enumerate}
As an aside, recall that we are dealing with extensions (and not with the more restricted root extensions). Indeed, if $c_u$ is a reticulation in $\cN_{t-1}$ and a vertex with in-degree one and out-degree one in $R_{c_u}'$, then $R_{c_u}$ is an extension and not a root extension. Now, regardless of which of (R1)--(R4) applies, let $$\cR'_t=(\cR'_{t-1}-\{R'_u,R'_{c_u},R'_{c_{u'}}\})\cup\{R_u,R_{c_u},R_{c_{u'}}\}.$$  It is easily checked that $\cR'_t$ is an extension of $\cD'$ in $\cN_t$ with
$$|E_{\cN_{t-1}}-E_{\cR'_{t-1}}|= |E_{\cN_{t}}-E_{\cR'_{t}}|,$$
and thus,
\begin{eqnarray}
\frac 1 2 m_\tc(\cN,\cN')&\leq&\frac 1 2 (|E_{\cN_0}-E_{\cR'_0}|+ |E_{\cN_{t}}-E_{\cR'_{t}}|)\nonumber\\
&=&\frac 1 2 (|E_{\cN_0}-E_{\cR'_0}|+ |E_{\cN_{t-1}}-E_{\cR'_{t-1}}|)\nonumber\\
&<& w(\sigma_1)+w(\sigma_2)=w(\sigma).\nonumber
\end{eqnarray}

We complete the proof of the lemma by assuming  that $e\in\cM'_{t-1}$. 
Let $e_{\cD'}=(u_{\cD'},v_{\cD'})$ be the unique edge in $\cD'$ that $e$ corresponds to. If $u_{\cD'}$ has in-degree two in $\cD'$, let  $p_{\cD'}$ and $p'_{\cD'}$ be the two parents of $u_{\cD'}$. Furthermore, if $u_{\cD'}$ has out-degree two, let $v'_{\cD'}$ be the child of $u_{\cD'}$ that is not $v_{\cD'}$ and, if $v'_{\cD'}$ is a reticulation, let $p''_{\cD'}$ be the parent of $v'_{\cD'}$ that is not $u_{\cD'}$.  Since $\cD'$ is tree-child, 
observe that each of $p_{\cD'}$, $p'_{\cD'}$, and $p''_{\cD'}$ has, if it exists, in-degree at most one, and that there exists a  directed path from each of $p_{\cD'}$, $p'_{\cD'}$, and $p''_{\cD'}$ to a leaf in $\cD'$ that does not traverse a reticulation. Lastly, since $\cD'$ is tree-child, at most one of $u_{\cD'}$, $v_{\cD'}$, and $v'_{\cD'}$ is a reticulation. Noting that $u_{\cD'}\ne\rho$, because $u\ne\rho$ by the definition of SNPR$^\pm$, we next obtain a digraph $\cD$ from $\cD'$ in one of the following five ways, which are illustrated in Figure~\ref{fig:D-cases}.

\begin{figure}[t]
\center
\scalebox{0.87}{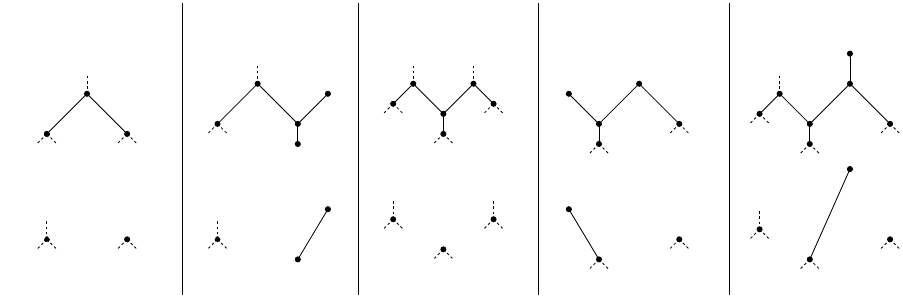}
\caption{The setup of $\cD'$ and $\cD$ in all five cases (D1)--(D5) in the construction of $\cD$ from $\cD'$ as detailed in the proof of Lemma~\ref{l:mtc<dtc}. Dashed edges may or may not exist in $\cD'$ and $\cD$. Vertex labels $y_1$, $y_2$, and $y_3$ are only used to clarify the figure.}
\label{fig:D-cases}
\end{figure}

\begin{enumerate}
\item [(D1)]Suppose that $u_{\cD'}$ has in-degree zero and out-degree two (resp.\ in-degree one and out-degree two), and that neither $v_{\cD'}$ nor $v'_{\cD'}$ is a reticulation. Then obtain $\cD$ from $\cD'$ by deleting $e_{\cD'}$ and deleting (resp.\ suppressing) $u_{\cD'}$.
\item [(D2)]Suppose that $u_{\cD'}$ has in-degree zero and out-degree two (resp.\ in-degree one and out-degree two), and that $v_{\cD'}$ is a reticulation. Then obtain $\cD$ from $\cD'$ by deleting $e_{\cD'}$, suppressing $v_{\cD'}$, and deleting (resp.\ suppressing)~$u_{\cD'}$.
\item [(D3)]Suppose that  $u_{\cD'}$ is a reticulation. Then obtain $\cD$ from $\cD'$ by applying the following three steps in order. First, delete $(p_{\cD'},u_{\cD'})$, suppress $u_{\cD'}$, and delete the resulting edge $(p'_{\cD'},v_{\cD'})$. Second, if  $p_{\cD'}$  has in-degree zero and out-degree two (resp.\ in-degree one and out-degree two) in $\cD'$, delete (resp.\ suppress) $p_{\cD'}$. Third, if  $p'_{\cD'}$  has in-degree zero and out-degree two (resp.\ in-degree one and out-degree two) in $\cD'$, delete (resp.\ suppress) $p'_{\cD'}$.
\item [(D4)]Suppose that $u_{\cD'}$ has in-degree zero and out-degree two, and that $v'_{\cD'}$ is a reticulation. Then obtain $\cD$ from $\cD'$ by deleting $e_{\cD'}$ and $u_{\cD'}$, and suppressing $v'_{\cD'}$.
\item [(D5)]Suppose that $u_{\cD'}$ has in-degree one and out-degree two, and that $v'_{\cD'}$ is a reticulation. Then obtain $\cD$ from $\cD'$ by deleting $e_{\cD'}$, suppressing $u_{\cD'}$, deleting $(p''_{\cD'}, v'_{\cD'})$, suppressing $v'_{\cD'}$, and if $p''_{\cD'}$ has in-degree zero and out-degree two (resp.\ in-degree one and out-degree two) in $\cD'$, deleting (resp.\ suppressing) $p''_{\cD'}$.
\end{enumerate}

\noindent By construction, it follows that  $\cD$  neither contains any vertex with in-degree zero and out-degree one except for $\rho$ nor a vertex with in-degree one and out-degree one. Hence, $\cD$ is a collection of leaf-labelled acyclic digraphs whose union of leaf sets is~$X$. We next show that $\cD$ is an agreement tree-child digraph for $\cN_0$ and $\cN_t$.

\begin{sublemma}\label{agreement}
$\cD$ is an agreement digraph for $\cN_0$ and $\cN_t$.
\end{sublemma}

\begin{proof}
Since $\cD$ is a collection of leaf-labelled acyclic digraphs whose union of leaf sets is $X$, Properties (i) and (ii) in the definition of a phylogenetic digraph are satisfied. Observe that in each of (D1)--(D5), $\cD$ is obtained from $\cD'$ by an ordered sequence $S$ of edge deletions, and vertex suppressions and deletions. Furthermore, by construction, a vertex is only suppressed (resp.\ deleted) if it has in-degree one and out-degree one (resp.\ in-degree zero and out-degree one) after an incident edge has been deleted. Following the order of operations in $S$, obtain an embedding $\cM_0$ of $\cD$ in $\cN_0$ from $\cM_0'$ as follows. For each edge $f_{\cD'}$ that is deleted in obtaining $\cD$ from $\cD'$, delete each non-terminal vertex of the  directed path in $\cM_0'$ that corresponds to $f_{\cD'}$ and, for each vertex that is deleted in obtaining $\cD$ from $\cD'$, delete the corresponding vertex in $\cM_0'$ and each resulting vertex that has in-degree zero and out-degree one (relative to the embedding) until no such vertex exists.\ As $\cD'$ is a phylogenetic digraph of $\cN_0$, it follows from the construction that $\cM_0$ is an embedding of $\cD$ in~$\cN_0$ and that the elements in $\cM_0$ are pairwise vertex disjoint in $\cN_0$. Thus, Property (iii) in the definition of a phylogenetic digraph is  satisfied and $\cD$ is a phylogenetic digraph of $\cN_0$. 

We complete the proof by showing that there also exists an embedding $\cM_t$ of $\cD$ in~$\cN_t$.  Obtain $\cM_t$ from $\cM_{t-1}'$ by applying the following two steps. First, if there exists an edge $f$ in $\cM_{t-1}'$ that corresponds to the edge $(p_{u'},c_{u'})$ in $\cN_{t-1}$, then subdivide~$f$ with a new vertex $u'$. Second, following again the order of operations in~$S$, for each edge $f_{\cD'}$ that is deleted in obtaining $\cD$ from $\cD'$, delete each non-terminal vertex of the directed path in $\cM_{t-1}'$ that corresponds to $f_{\cD'}$  
and, for each vertex that is deleted in obtaining $\cD$ from $\cD'$, delete the corresponding vertex in $\cM_{t-1}'$ and each resulting vertex that has in-degree zero and out-degree one (relative to the embedding) until no such vertex exists.
To see that $\cM_t$ is indeed an embedding of $\cD$ in $\cN_t$, recall that $\cN_t$ can be obtained from $\cN_{t-1}$ by deleting~$e$, suppressing $u$, subdividing $(p_{u'},c_{u'})$ with a new vertex $u'$, and adding a new edge~$(u',v)$. 
Since~$e_{\cD'}$ is deleted and $u_{\cD'}$ is either suppressed or deleted in each of (D1)--(D5), it now follows from the construction and the fact that $\cM_{t-1}'$ satisfies Property (iii) in the definition of a phylogenetic digraph that $\cM_t$ is an embedding of $\cD$ in $\cN_t$ and that the elements of $\cM_t$ are also pairwise vertex disjoint in $\cN_t$. Thus, Property (iii) in the definition of a phylogenetic digraph is satisfied, and $\cD$ is a phylogenetic digraph of $\cN_t$.
\end{proof}

\begin{sublemma}\label{stupid-thing}
$\cD$ is tree-child.
\end{sublemma}

\begin{proof}
Assume $\cD$ is not tree-child. Since $\cD'$ is tree-child, it follows from the construction of $\cD$ that $v'_{\cD'}$ is a reticulation in $\cD'$ and $\cD$. However, if $v'_{\cD'}$ is  a reticulation in $\cD'$, then (D4) or (D5) applies and in each case one of the reticulation edges that are directed into $v'_{\cD'}$ is deleted. Thus, $v'_{\cD'}$ is not a reticulation in $\cD$, a contradiction.
\end{proof}

\noindent It now follows from~\eqref{agreement} and~\eqref{stupid-thing} that $\cD$ is an agreement tree-child digraph for $\cN_0$ and $\cN_t$.

\begin{sublemma}\label{almost-last}
There exists an extension $\cR_t$ of $\cD$ in $\cN_t$ such that 
$$|E_{\cN_t}-E_{\cR_t}|\leq |E_{\cN_{t-1}}-E_{\cR'_{t-1}}|+2.$$
\end{sublemma}

\begin{proof}
To ease reading, we view $\cR'_{t-1}$ as a collection of edges in $\cN_{t-1}$ and describe the construction of $\cR_t$ from $\cR'_{t-1}$ by edge deletions and additions only. Let $P$ be the directed path in $\cN_{t-1}$ that $e_{\cD'}$ corresponds to. Clearly $e$ is an edge of $P$. Furthermore, let $(s,t)$ be the first edge on $P$, and let $E_s$ be the path extension of the subpath of $P$ from $s$ to $v$. Observe that, if $s=u$, then $s$ is a tree vertex in~$\cN_{t-1}$. Hence, in this case, $R_{c_u}'=R_u'$ and $(u,c_u)\in R_u'$. On the other hand, if $s\ne u$, then $(u,c_u)$ is an edge in $E_{\cN_{t-1}}-E_{\cR'_{t-1}}$. 

We next obtain $\cR_t$ from $\cR_{t-1}'$. Intuitively, we construct digraphs $R_u$ and $R_{c_u}$ from $R_u'$ and $R_{c_u}'$, respectively, such that $R_u$ and $R_{c_u}$ are extensions of elements in $\cD$ in $\cN_t$. As we will see, some of the edges in $R_u'$ and $R_{c_u}'$ that are edges of $\cM_{t-1}'$ become edges of $R_u$ and $R_{c_u}$, respectively, that are not edges of the embedding that underlies $\cR_t$. Now suppose that $\cD$ has been obtained from $\cD'$ by applying the construction as detailed in (D1) or (D2). Obtain $R_u$ and $R_{c_u}$ from $R_u'$ and $R_{c_u}'$, respectively, in one of the following four ways:
\begin{enumerate}
\item [(R1$'$)] Suppose that $s=u$ and $(p_u,u)\in R_u'$. Obtain $R_u$ from $R_u'$ by deleting $(p_u,u)$, $(u,c_u)$, and $e$,  and adding $(p_u,c_u)$. 
\item [(R2$'$)] Suppose that $s=u$ and $(p_u,u)\notin R_u'$. Obtain $R_u$ from $R_u'$ by deleting $e$ and $(u,c_u)$.
\item [(R3$'$)] Suppose that $s\ne u$ and $R_u'=R_{c_u}'$. Obtain $R_u$ from $R_u'$ by deleting $(s,t)$, $(p_u,u)$, and $e$, and  if $t\ne u$, adding $(p_u,c_u)$.
\item [(R4$'$)] Suppose that $s\ne u$ and $R_u'\ne R_{c_u}'$. First obtain $R_u$ from $R_u'$ by deleting each edge in $E_s$. Second if $t=u$, set $R_{c_u}=R_{c_u}'$. Otherwise if $t\ne u$, obtain $R_{c_u}$ from $R_{c_u}'$ by adding $(p_u,c_u)$ and adding each edge in $E_s$ except for $(s,t)$, $(p_u,u)$, and $e$.
\end{enumerate}
While (R3$'$) and (R4$'$) are similar in flavour, they are different in the sense that, if $R_u'\ne R'_{c_u}$, then certain edges in $R_u'$ that lie on a path extension of a subpath of $P$ are not edges of $R_u$ and instead get added to $R_{c_u}$. For an illustration of (R4$'$), see the left-hand side of Figure~\ref{fig:R-cases1}.

\begin{figure}[t]
    \centering
\scalebox{0.87}{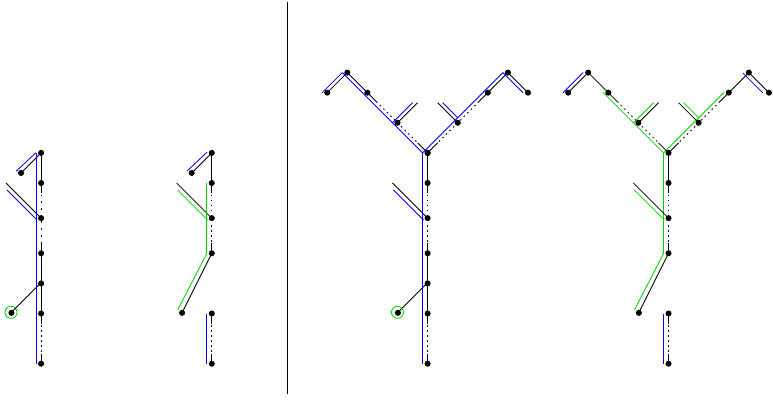}
    \caption{An illustration of (R4$'$) and (R6$'$) as described in the proof of~\eqref{almost-last}. Blue indicates edges and vertices of $R_u'$ and $R_u$, and green indicates edges and vertices of $R_{c_u}'$ and $R_{c_u}$.}
    \label{fig:R-cases1}
\end{figure}

Next, suppose that $\cD$ has been obtained from $\cD'$ by applying the construction as detailed in (D3). As $u_{\cD'}$ has in-degree two and corresponds to $s$ in $\cN_{t-1}$, observe that $s\ne u$. Let $P_1$ (resp.\ $P'_1$) be the directed path in $\cN_{t-1}$ that the edge $(p_{\cD'},u_{\cD'})$ (resp.\ $(p'_{\cD'},u_{\cD'})$) in $\cD'$ corresponds to. Furthermore, let $(s_1,t_1)$ (resp.\ $(s_1',t_1')$) be the first edge on $P_1$ (resp.\ $P_1'$). Since $\cD'$ is tree-child, neither $s_1$ nor~$s_1'$ is a reticulation in $\cN_{t-1}$. Similar to the definition of $E_s$, let $E_1$ (resp.\ $E_1'$) be the path extension of $P_1$ (resp.\ $P_1'$). Finally, obtain $R_u$ and $R_{c_u}$ from $R_u'$ and $R_{c_u}'$, respectively, in one of the following two ways:
\begin{enumerate}
\item [(R5$'$)] Suppose that $R_u'=R_{c_u}'$. Obtain $R_u$ from $R_u'$ by deleting $(s_1,t_1)$,  $(s_1',t_1')$, $(p_u,u)$, and $e$, and adding $(p_u,c_u)$. 
\item [(R6$'$)] Suppose that $R_u'\ne R_{c_u}'$. First obtain $R_u$ from $R_u'$ by deleting each edge in $E_1$, $E_1'$, and $E_s$. Second obtain $R_{c_u}$ from $R_{c_u}'$ by adding each edge  in $E_1$ except for $(s_1,t_1)$, adding each edge in $E_1'$ except for $(s_1',t_1')$, adding each edge in $E_s$ except for $(p_u,u)$ and $e$, and adding $(p_u,c_u)$. The construction is shown on the right-hand side of Figure~\ref{fig:R-cases1}.
\end{enumerate}

Lastly, suppose that $\cD$ has been obtained from $\cD'$ by applying the construction as detailed in (D4) or (D5). Let $Q_1$ (resp.\ $Q'_1$) be the directed path in $\cN_{t-1}$ that the edge $(u_{\cD'},v'_{\cD'})$ (resp.\ $(p''_{\cD'},v'_{\cD'})$) in $\cD'$ corresponds to. Furthermore, let $(s_1,t_1)$ (resp.\ $(s_1',t_1')$) be the first edge on $Q_1$ (resp.\ $Q_1'$). Note that $s_1=s$ and, if $s=u$, then $c_u=t_1$. Say first that $\cD$ has been obtained from $\cD'$ by applying the construction as detailed in (D4). Let $F$ be the subset of edges of $\cN_{t-1}$ that lie on a directed path of  $R_u'$  that ends at $s$. Observe that each edge in $F$ is contained in $E_{\cR'_{t-1}}-E_{\cM'_{t-1}}$. Now obtain $R_u$ and $R_{c_u}$ from $R_u'$ and $R_{c_u}'$, respectively, by applying one of (R1$'$) and (R2$'$) if $s=u$, or by applying one of the following two ways if $s\ne u$:
\begin{enumerate}
\item [(R7$'$)]Suppose that $s\ne u$ and $R_u'= R_{c_u}'$. Obtain $R_u$ from $R_u'$ by deleting $(s_1,t_1)$, $(p_u,u)$, and $e$, and adding the edge $(p_u,c_u)$. 
\item  [(R8$'$)]Suppose that $s\ne u$ and $R_u'\ne R_{c_u}'$. First obtain $R_u$ from $R_u'$ by deleting each edge in $E_s$ and $F$, and deleting $(s_1,t_1)$. Second obtain $R_{c_u}$ from $R_{c_u}'$ by adding each edge in $E_s$ except for $(p_u,u)$ and $e$, adding $(p_u,c_u)$, and adding each edge in $F$. See the left-hand side of Figure~\ref{fig:R-cases2} for an illustration.
\end{enumerate}

\begin{figure}[t]
    \centering
\scalebox{0.87}{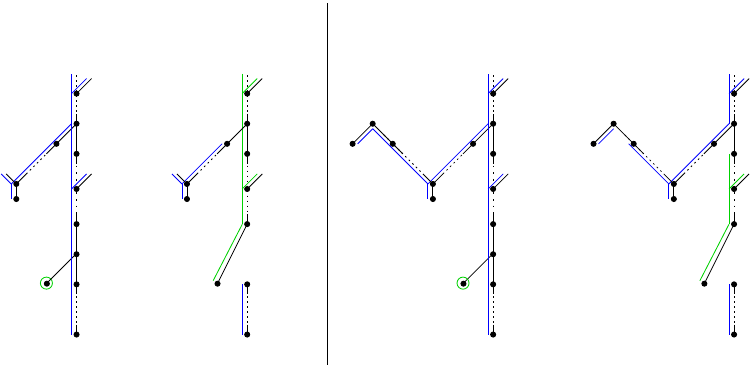}
    \caption{An illustration of (R8$'$) and (R11$'$) as described in the proof of~\eqref{almost-last}. Blue indicates edges and vertices of $R_u'$ and~$R_u$, and green indicates edges and vertices of $R_{c_u}'$ and $R_{c_u}$.}
    \label{fig:R-cases2}
\end{figure}

\noindent On the other hand, if $\cD$ has been obtained from $\cD'$ by applying the construction as detailed in (D5), then obtain $R_u$ and $R_{c_u}$ from $R_u'$ and $R_{c_u}'$, respectively, in one of the following three ways:
\begin{enumerate}
\item  [(R9$'$)]Suppose that $s=u$. Obtain $R_u$ from $R_u'$ by deleting $(s_1',t_1')$, $(p_u,u)$, $(u,c_u)$, and $e$, and adding $(p_u,c_u)$.
\item  [(R10$'$)]Suppose that $s\ne u$ and $R_u'=R_{c_u}'$. Obtain $R_u$ from $R_u'$ by deleting $(s_1',t_1')$, $(s,t)$, $(p_u,u)$, and $e$, and, if $t\ne u$, then adding $(p_u,c_u)$.
\item [(R11$'$)] Suppose that $s\ne u$ and $R_u'\ne R_{c_u}'$. First obtain $R_u$ from $R_u'$ by deleting $(s_1',t_1')$ and each edge in $E_s$. Second, if $t=u$, set $R_{c_u}=R_{c_u}'$. Otherwise, if $t\ne u$  obtain $R_{c_u}$ from $R_{c_u}'$ by adding each edge in $E_s$ except for $(s,t)$, $(p_u,u)$, and $e$, and adding $(p_u,c_u)$. See the right-hand side of Figure~\ref{fig:R-cases2} for an illustration.
\end{enumerate}

Finally, let $\cR_t=(\cR_{t-1}'-\{R_u',R_{c_u}',R_{c_{u'}}'\})\cup\{R_u,R_{c_u},R_{c_{u'}}\}$. Since $\cR_{t-1}'$ is an extension of $\cD'$ in $\cN_{t-1}$, a careful check shows that $\cR_t$ is an extension of $\cD$ in $\cN_t$.

Now, let $C'=E_{\cN_{t-1}}-E_{\cR'_{t-1}}$, and let $C=E_{\cN_t}-E_{\cR_t}$. To complete the proof of~\eqref{almost-last}, we compare the number of edges in $C'$ with the number of edges in~$C$. First, observe that, if $(p_{u'},c_{u'})\in C'$, then  $(p_{u'},u')\in C$ and $(u',c_{u'})\notin C$. Furthermore, if $(p_{u'},c_{u'})\notin C'$, then neither $(p_{u'},u')$ nor $(u',c_{u'})$ is in $C$. We next list the edges that are in $C'$ but not in $C$ and vice versa for each of (R1$'$)--(R11$'$). While  $C'-C$ contains edges in  $\cN_{t-1}$ that are not edges in $\cN_t$, the set $C-C'$ contains edges in $\cN_t$ that are not edges in $\cN_{t-1}$. Thus, edges that are common to $\cN_{t-1}$ and $\cN_t$ and common to $C'$ and $C$ are not considered in the following table. Moreover, regardless of which of (R1$'$)--(R11$'$) applies, we note that $C'-C$ may or may not contain $(p_{u'}, c_{u'})$ and $C-C'$ may or may not contain $(p_{u'}, u')$. However, $C'-C$ contains $(p_{u'}, c_{u'})$ if and only if $C-C'$ contains $(p_{u'}, u')$, and so we have also omitted in the table the possibility that $C'-C$ may contain $(p_{u'}, c_{u'})$ and the possibility that $C-C'$ may contain $(p_{u'}, u')$.

\begin{center}
\renewcommand{\arraystretch}{1.25}
\begin{tabular}{p{85pt}|p{110pt}| p{120pt}}
& $(C’-C)-\{(p_{u'}, c_{u'})\}$ & $(C-C’)-\{(p_{u'}, u')\}$\\ \midrule\midrule
(R1$'$) & empty & $(u',v)$ \\ 
(R2$'$) & $(p_u,u)$ & $(p_u,c_u)$, $(u',v)$ \\ 
(R3$'$) and $t=u$ & $(u,c_u)$ & $(p_u,c_u)$, $(u',v)$ \\
(R3$'$) and $t\ne u$ & $(u,c_u)$ & $(s,t)$, $(u',v)$ \\
(R4$'$) and $t=u$ & $(u,c_u)$ & $(p_u,c_u)$, $(u',v)$ \\
(R4$'$) and $t\ne u$ & $(u,c_u)$ & $(s,t)$, $(u',v)$ \\ 
(R5$'$) and (R6$'$) & $(u,c_u)$ & $(s_1,t_1)$, $(s_1',t_1')$,  $(u',v)$ \\ 
(R7$'$) and (R8$'$) & $ (u,c_u)$ & $(s_1,t_1)$, $(u',v)$ \\ 
(R9$'$) & empty & $(s_1',t_1')$, $(u',v)$ \\ 
(R10$'$) and $t=u$ & $(u,c_u)$ & $(s_1',t_1')$, $(p_u,c_u)$, $(u',v)$ \\ 
(R10$'$) and $t\ne u$ & $(u,c_u)$ & $(s_1',t_1')$, $(s,t)$, $(u',v)$ \\ 
(R11$'$) and $t=u$ & $(u,c_u)$ & $(s_1',t_1')$, $(p_u,c_u)$, $(u',v)$ \\
(R11$'$) and $t\ne u$ & $(u,c_u)$ & $(s_1',t_1')$, $(s,t)$, $(u',v)$ \\ 
\end{tabular}
\end{center}
\noindent Since $ |C-C'|\leq |C'-C|+2$ in all cases, this completes the proof of~\eqref{almost-last}.
\end{proof}

\begin{sublemma}\label{last}
There exists an extension $\cR_0$ of $\cD$ in $\cN_0$ such that 
$$|E_{\cN_0}-E_{\cR_0}|\leq |E_{\cN_{0}}-E_{\cR'_{0}}|+2.$$
\end{sublemma}

\begin{proof}
Again, to ease reading, we view $\cR'_{0}$ simply as a collection of edges in $\cN_{0}$ and describe the construction of $\cR_0$ from $\cR'_{0}$ by edge deletions only. Let $P$ be the directed path in $\cN_0$ that $e_{\cD'}$ corresponds to, and let $(s,t)$ be the first edge on~$P$. Let $R_s'$ be the element in $\cR_0'$ that contains $s$. We next construct an extension $\cR_0$ of $\cD$ in $\cN_0$ by modifying $R_s'$. This construction is similar to the constructions described in proof of~\eqref{almost-last}, but much less involved. First, suppose that $\cD$ has been obtained from $\cD'$ by applying (D1), (D2), or (D4). Then
\begin{enumerate}
\item [(R1$''$)] obtain $R_s$ from $R_s'$ by deleting $(s,t)$.
\end{enumerate}
Second, suppose that $\cD$ has been obtained from $\cD'$ by applying (D3). Recall that $s$ is a reticulation.  Let $P_1$ (resp.\ $P'_1$) be the directed path in $\cN_{0}$ that the edge $(p_{\cD'},u_{\cD'})$ (resp.\ $(p'_{\cD'},u_{\cD'})$) in $\cD'$ corresponds to. Furthermore, let $(s_1,t_1)$ (resp.\ $(s_1',t_1')$) be the first edge on $P_1$ (resp.\ $P_1'$). Then
\begin{enumerate}
\item [(R2$''$)]obtain $R_s$ from $R_s'$ by deleting $(s_1,t_1)$ and $(s_1',t_1')$. 
\end{enumerate}
Third, suppose that $\cD$ has been obtained from $\cD'$ by applying (D5). Let $P_1$  be the directed path in $\cN_{0}$ that the edge $(p''_{\cD'},v'_{\cD'})$ in $\cD'$ corresponds to. Furthermore, let $(s_1,t_1)$ be the first edge on $P_1$. Then
\begin{enumerate}
\item [(R3$''$)]obtain $R_s$ from $R_s'$ by deleting $(s,t)$ and $(s_1,t_1)$. 
\end{enumerate}
Now, let $\cR_0=(\cR_{0}'-\{R_s'\})\cup\{R_s\}$. Since $\cR_{0}'$ is an extension of $\cD'$ in $\cN_{0}$, a careful check shows that $\cR_0$ is an extension of $\cD$ in $\cN_0$. Moreover, since each of \mbox{(R1$''$)--(R3$''$)} deletes at most two edges in obtaining $R_s$ from $R_s'$, it follows that~\eqref{last} holds.
\end{proof}

Finally, by combining Equation~\eqref{induction} with~\eqref{agreement}and~\eqref{last}, we get
\begin{eqnarray}
\frac 1 2 m_\tc(\cN,\cN')&\leq&\frac 1 2 (|E_{\cN_0}-E_{\cR_0}|+ |E_{\cN_{t}}-E_{\cR_{t}}|)\nonumber\\
&\leq&\frac 1 2 (|E_{\cN_0}-E_{\cR'_0}|+2+ |E_{\cN_{t-1}}-E_{\cR'_{t-1}}|+2)\nonumber\\
&\leq&\frac 1 2 (|E_{\cN_0}-E_{\cR'_0}|+ |E_{\cN_{t-1}}-E_{\cR'_{t-1}}|)+\frac 1 2\cdot 4\nonumber\\
&\leq& w(\sigma_1)+w(\sigma_2)\nonumber\\
&=&w(\sigma).\nonumber
\end{eqnarray}
This completes the proof of the lemma.
\end{proof}

\noindent {\it Proof of Theorem~\ref{t:main}.} The theorem follows from Lemmas~\ref{l:dtc<mtc} and~\ref{l:mtc<dtc}.\qed

The following result shows how the rSPR distance between two phylogenetic trees can be computed exactly within the framework of agreement digraphs. In particular, it shows that agreement digraphs generalise agreement forest.

\begin{proposition}\label{prop:rspr}
Let $\cT$ and $\cT'$ be two phylogenetic $X$-trees. Then $$d_\rSPR(\cT,\cT')=\frac 1 2 d_\tc(\cT,\cT')=\frac 1 2 m_\tc(\cT,\cT').$$
\end{proposition}

\begin{proof}
The first equality follows from Bordewich et al.~\cite[Proposition 7.1]{bordewich17} and the fact that each SNPR$^\pm$ contributes two to the weight of any SNPR sequence connecting $\cT$ and $\cT'$. Moreover, to establish the second equality, let $$\sigma=(\cT=\cT_0,\cT_1,\cT_2,\ldots,\cT_t=\cT')$$ be an SNPR sequence connecting $\cT$ and $\cT'$. Then it follows from Lemma~\ref{l:dtc<mtc} and a careful inspection of the proof of Lemma~\ref{l:mtc<dtc} when applied to two phylogenetic $X$-trees that (D1) and, consequently, (R1$'$)--(R4$'$) and (R1$'$) always apply. Hence, the last set of inequalities in the proof of Lemma~\ref{l:mtc<dtc} can be replaced with
\begin{eqnarray}
m_\tc(\cT,\cT')&\leq& |E_{\cT_0}-E_{\cR_0}|+ |E_{\cT_{t}}-E_{\cR_{t}}|\nonumber\\
&\leq&|E_{\cT_0}-E_{\cR'_0}|+1+ |E_{\cT_{t-1}}-E_{\cR'_{t-1}}|+1\nonumber\\
&\leq& w(\sigma_1)+w(\sigma_2)\nonumber\\
&=&w(\sigma),\nonumber
\end{eqnarray} 
where $\sigma_1=(\cT_0, \cT_1, \cT_2, \ldots, \cT_{t-1})$ and $\sigma_2=(\cT_{t-1}, \cT_t)$.
\end{proof}

The next proposition shows that the bound given in Lemma~\ref{l:mtc<dtc} is essentially tight.

\begin{figure}[t]
\center
\scalebox{0.87}{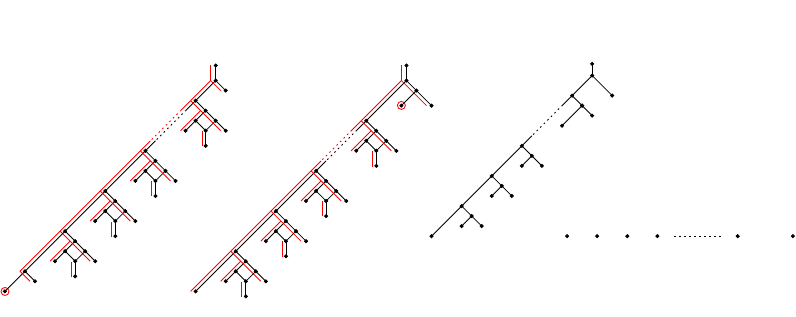}
\caption{Two phylogenetic networks $\cN_\ell$ and $\cN_\ell'$ on $n=3\ell$ leaves with $\ell>1$, an agreement tree-child digraph $\cD_\ell$ for $\cN_\ell$ and~$\cN_\ell'$, an extension of $\cD_\ell$ in $\cN_\ell$  and an extension of $\cD_\ell$ in $\cN_\ell'$ indicated in red.\ This example shows that the bound given in Lemma~\ref{l:mtc<dtc} is essentially tight. For details, see the proof of Proposition~\ref{prop:sharp}.}
\label{fig:sharp}
\end{figure}

\begin{proposition}\label{prop:sharp}
For any  integer $\ell$ with $\ell>1$, there exist two tree-child networks $\cN_\ell$ and $\cN'_\ell$ on $3\ell$ leaves such that $\frac 1 2 m_\tc(\cN_\ell,\cN'_\ell)+1=d_\tc(\cN_\ell,\cN'_\ell)$.
\end{proposition}

\begin{proof}
Let $\ell$ be an integer with $\ell>1$. Consider the two tree-child networks $\cN_\ell$ and~$\cN'_\ell$ that are shown in Figure~\ref{fig:sharp}. Each of $\cN_\ell$ and $\cN'_\ell$ has $3\ell$ leaves. Moreover, the agreement tree-child digraph $\cD_\ell$ for $\cN_\ell$ and $\cN'_\ell$ that is also shown in Figure~\ref{fig:sharp} has cut size  $2\ell-1$ in each of $\cN_\ell$ and $\cN_\ell'$. Thus, $m_\tc(\cN,\cN')\leq 4\ell-2$. We now show that $m_{\rm tc}(\cN, \cN')=4\ell-2$. Assume that  $m_\tc(\cN,\cN')< 4\ell-2$. Then there exists an agreement tree-child digraph $\cD_\ell^*=\{D_\rho,D_1,D_2,\ldots,D_k\}$ whose cut size in  $\cN_\ell$ or $\cN'_\ell$ is strictly less than $2\ell-1$. Since, for each $j\in\{1,2,\ldots,\ell-1\}$, $\cN_\ell$ displays the two triples $(3j,3j+1,3j+2)$ and $(3j+1,3j+2,3j)$ whereas $\cN_\ell'$ only displays the triple $(3j,3j+2,3j+1)$ and no other triple involving $3j$, $3j+1$, and $3j+2$, a careful check shows that $\cD^*_\ell$ contains an element that is not a phylogenetic tree. 
To see this, note that if $\cD^*_{\ell}$ only consists of phylogenetic trees, then each $j\in\{1,2,\ldots,\ell-1\}$, $\cN_\ell$ contributes two to the cut size of $\cD^*_{\ell}$ in $\cN_{\ell}$ and two to the cut size of $\cD^*_{\ell}$ in $\cN_{\ell}'$.
Thus, there exists an element $D_i$ in $\cD^*_\ell$ for some $i\in\{\rho,1,2,\ldots,k\}$ that contains a vertex $v$ of in-degree two and out-degree one. Moreover, as $\cD^*_{\ell}$ is an agreement digraph of $\cN_{\ell}$ and $\cN'_{\ell}$, the child of $v$ is $j$ for some $j\in\{4,7,\ldots,3\ell-5\}$. First, assume that there exists a vertex $u$ in $D_i$ and two edge-disjoint directed paths $P$ and $P'$ from $u$ to $v$. Since $\cD^*_\ell$ is tree-child, at least one of $P$ and $P'$ contains a vertex $w$ such that the edge $(w,v)$ with $w\ne u$ exists. Furthermore, as $D_i$ can be embedded in $\cN_\ell$ and $P$ and $P'$ are edge disjoint, we may assume, that the child of $w$ that is not $v$ is $j-1$ or $j+1$. In either case, it is easily seen that there is no embedding of $D_i$ in $\cN_\ell'$, a contradiction. Thus, we may assume that there exist two vertices $u$ and $u'$ with in-degree zero in $D_i$ and  directed paths from each of $u$ and $u'$ to $v$ whose only common vertex is $v$. As $D_i$ can be embedded in $\cN_{\ell}$, we may assume without loss of generality that  the child of $u$ is $j+1$ or $j-1$ which leads us to the same contradiction as in the previous case because there is no such embedding of $D_i$ in $\cN'_{\ell}$.
Hence, there exists no agreement tree-child digraph whose cut size in  $\cN_\ell$ or $\cN'_\ell$ is strictly less than $2\ell-1$. It now follows that 
\begin{equation}\label{eq:sharp-one}
m_\tc(\cN,\cN')=4\ell-2.
\end{equation}

Turning to $d_\tc(\cN_\ell,\cN'_\ell)$, observe that there exists a tree-child SNPR sequence connecting $\cN_\ell$ and $\cN'_\ell$ that prunes and regrafts the leaves $1,4,7,\ldots,3\ell-2$ in order. Hence, 
\begin{equation}\label{eq:sharp-two}
d_\tc(\cN_\ell,\cN'_\ell)\leq2\ell.
\end{equation}

By combining Lemma~\ref{l:mtc<dtc} with Equations~\eqref{eq:sharp-one} and~\eqref{eq:sharp-two}, we have
$$2\ell-1=\frac 1 2 m_\tc(\cN_\ell,\cN_\ell')\leq d_\tc(\cN_\ell,\cN'_\ell)\leq 2\ell$$
which, in turn, implies that $d_\tc(\cN_\ell,\cN'_\ell)\in\{2\ell-1,2\ell\}$. Since each of $\cN_\ell$ and $\cN'_\ell$ has $\ell-1$ reticulations, the weight of any SNPR sequence connecting $\cN_\ell$ and $\cN'_\ell$ is even. Thus, $d_\tc(\cN_\ell,\cN'_\ell)=2\ell$, thereby establishing the proposition.
\end{proof}

\section{Concluding remarks}\label{sec:conclusion}

In this paper, we have taken a step towards approximating the tree-child SNPR distance $d_\tc(\cN,\cN')$ between two tree-child networks $\cN$ and $\cN'$. By introducing phylogenetic digraphs and their extensions, thereby generalising agreement forests for two phylogenetic trees to two phylogenetic networks, we have shown that $d_\tc(\cN,\cN')$ is tightly bounded from above and below within small constant factors of $m_\tc(\cN,\cN')$. A possible next step is the development of an algorithm to compute $m_\tc(\cN,\cN')$. Due to the intricacies of phylogenetic digraphs and their embeddings, this is a major challenge. In addition, it immediately follows from Proposition~\ref{prop:rspr} and the NP-hardness of computing the rSPR distance between two phylogenetic trees $\cT$ and $\cT'$~\cite{bordewich05} that computing $m_\tc(\cN,\cN')$ is also NP-hard. Since it seems natural to assume that any algorithm for computing $m_\tc(\cN,\cN')$  needs to repeatedly compute cut sizes, it would be interesting to investigate if the cut size of a given phylogenetic digraph for a phylogenetic network $\cN$ can be computed efficiently. In a different direction, the development of reductions and divide-and-conquer strategies for computing $m_\tc(\cN,\cN')$ is another avenue for future research. For example, in the context of phylogenetic trees, the introduction of the subtree and chain reduction has led to fixed-parameter tractable algorithms for computing the rSPR distance~\cite{bordewich05}. Do similar reductions exist for phylogenetic networks?

As mentioned in the introduction, we use two different weights for SNPR operations to make our approach work. Specifically, each SNPR$^+$ and SNPR$^-$ is weighted one and each SNPR$^\pm$ is weighted two. While these weights differ from the  uniform weights that are used for computing the rSPR distance between two phylogenetic trees, they are a consequence of how our generalisation from agreement forests to agreement digraphs and their embeddings works. Without going into detail, given two phylogenetic trees $\cT$ and $\cT'$ with $d_\rSPR(\cT,\cT')=k$, there exists an agreement forest $\cF$ for $\cT$ and $\cT'$ with $k+1$ components. Moreover, one can obtain $\cF$ from  $\cT$ (resp.\ $\cT'$) by deleting $k$ edges in $\cT$ (resp.\ $\cT'$) and suppressing vertices with in-degree one and out-degree one after each edge deletion. Intuitively, each rSPR operation is witnessed by an edge in $\cT$ and by an edge in $\cT'$. In the language of this paper, any agreement forest $\cF$ for $\cT$ and $\cT'$ has the property that its cut size in $\cT$ is equal to its cut size in $\cT'$ and, thus, $d_\rSPR(\cT,\cT')$ simply equates to the cut size of $\cF$ in one of the two trees.\ Now consider two tree-child networks $\cN$ and $\cN'$ such that $\cN'$ can be obtained from $\cN$ by a sequence $\sigma$ of only SNPR$^+$ operations. Suppose that the length of $\sigma$ is $k$. In this case, $\cN$ is an agreement tree-child digraph $\cD$ for $\cN$ and $\cN'$, the cut size of $\cD$ in $\cN$ is zero, and the cut size of $\cD$ in $\cN'$ is $k$.  More generally, for an arbitrary tree-child SNPR sequence that connects two tree-child networks $\cN$ and $\cN'$, each SNPR$^\pm$ contributes to the cut sizes of both trees, whereas each SNPR$^+$ and SNPR$^-$ only contributes to the cut size of one tree. Thus any approach for computing $d_\tc(\cN,\cN')$ that is based on cut sizes as defined in this paper (probably) needs to apply non-uniform weights to the different types of SNPR operations.  Ultimately, it would be interesting to investigate whether or not an approach exists for computing  $d_\tc(\cN,\cN')$ that allows for any combination of weights.

\subsection*{Acknowledgements}

We thank the two anonymous referees for their careful reading and constructive comments. All authors thank the New Zealand Marsden Fund for their financial support.~Part of this paper is based upon work supported by the National Science Foundation under Grant No. DMS-1929284 while the second and third authors were in residence at the Institute for Computational and Experimental Research in Mathematics in Providence, Rhode Island, US, during the {\it Theory, Methods, and Applications of Quantitative Phylogenomics} program.



\begin{thebibliography}{99}

\bibitem{allen01}
B. L. Allen and M. Steel (2001). Subtree transfer operations and their induced metrics on evolutionary trees. {\it Annals of Combinatorics}, 5:1--15.

\bibitem{phd}
C. Allen-Savietta (2020). {\it Estimating phylogenetic networks from concatenated sequence alignments.} PhD thesis, University of Wisconsin-Madison.

\bibitem{relaxed}
V. Ard\'evol Mart\'inez, S. Chaplick, S. Kelk, R. Meuwese, M. Mihal\'ak, and G. Stamoulis. Relaxed agreement forests. In:\ H. Fernau, S. Gaspers, and R. Klasing (Eds.), {\it SOFSEM 2024: Theory and Practice of Computer Science}, pp. 40--54, Springer.

\bibitem{atkins19}
R. Atkins and C. McDiarmid (2019). Extremal distances for subtree transfer operations in binary trees. {\it Annals of Combinatorics}, 23:1--26.

\bibitem{baroni05}
M. Baroni, S. Gr\"unewald, V. Moulton, and C. Semple (2005). Bounding the number of hybridisation events for a consistent evolutionary history. {\it Journal of Mathematical Biology},  51:171--182.

\bibitem{bordewich05}
M. Bordewich and C. Semple (2005). On the computational complexity of the rooted subtree prune and regraft distance. {\it Annals of Combinatorics}, 8:409--423.
  
\bibitem{bordewich17} 
M. Bordewich, S. Linz, and C. Semple (2017). Lost in space? Generalising subtree prune and regraft to spaces of phylogenetic networks. {\it Journal of Theoretical Biology}, 423:1--12.

\bibitem{cardona09}
G. Cardona, F. Rossell\'o, and G. Valiente (2009). Comparison of tree-child phylogenetic networks. {\it IEEE/ACM Transactions on Computational Biology and Bioinformatics}, 6:552--569.

\bibitem{chen15}
J. Chen, J-H. Fan, and S-H. Sze (2015). Parameterized and approximation algorithms for maximum agreement forest in multifurcating trees. {\it Theoretical Computer Science}, 562:496--512.

\bibitem{choy05}
C. Choy, J. Jansson, K. Sadakane, and W.-K. Sung (2005). Computing the maximum agreement of phylogenetic networks. {\it Theoretical Computer Science}, 335:93--107.

\bibitem{ding11}
Y. Ding, S. Gr\"unewald, P. J. Humphries (2011) On agreement forests. {\it Journal of Combinatorial Theory Series A}, 118:2059--2065.

\bibitem{doecker21}
J. D\"ocker, S. Linz, and C. Semple (2021). The display sets of binary normal and tree-child networks. {\it The Electronic Journal of Combinatorics}, 28, Paper 1.8.

\bibitem{erdos21}
P. L. Erd{\H{o}}s, A. Francis, and T. R. Mezei (2021). Rooted NNI moves and distance-$1$ tail moves on tree-based phylogenetic networks. {\it Discrete Applied Mathematics}, 294:205--213.

\bibitem{francis17}
A. Francis, K. T. Huber, V. Moulton, and T. Wu (2017). Bounds for phylogenetic network space metrics. {\it Journal of Mathematical Biology}, 76:1229--1248.

\bibitem{gambette17}
P. Gambette, L. van Iersel, M. Jones. M. Lafond, F. Pardi, and C. Scornavacca (2017). Rearrangement moves on rooted phylogenetic networks. {\it PLoS Computational Biology}, 13:e1005611.

\bibitem{hein96}
J. Hein, T. Jiang, L. Wang, and K. Zhang (1996). On the complexity of comparing evolutionary trees. {\it Discrete Applied Mathematics}, 71:153--169.

\bibitem{huber15}
K. T. Huber. S. Linz, V. Moulton, and T. Wu (2015). Spaces of phylogenetic networks from generalized nearest-neighbor interchange operations. {\it Journal of Mathematical Biology}, 72:699--725.

\bibitem{huber16}
K. T. Huber, V. Moulton, and T. Wu (2016). Transforming phylogenetic networks: Moving beyond tree space. {\it Journal of Theoretical Biology}, 404:30--39.

\bibitem{janssen21}
R. Janssen (2021). Heading in the right direction? Using head moves to traverse phylogenetic network space. {\it Journal of Graph Algorithms and Applications}, 25:263--310.

\bibitem{janssen24}
R. Janssen (2024). PhyloX: A Python package for complete phylogenetic network workflows. {\it Journal of Open Source Software}, 9:6427.

\bibitem{janssen18}
R. Janssen, M. Jones, P. L. Erd{\H{o}}s, L. van Iersel, and C. Scornavacca (2018). Exploring the tiers of rooted phylogenetic network space using tail moves. {\it Bulletin of Mathematical Biology}, 80:2177--2208.

\bibitem{janssen19}
R. Janssen and J. Klawitter (2019). Rearrangement operations on unrooted phylogenetic networks. {\it Theory and Applications of Graphs}, 22:1--31.

\bibitem{janssen-phd}
R. Janssen (2021). {\it Rearranging phylogenetic networks}. PhD thesis, Delft University of Technology.

\bibitem{jansson04}
J. Jansson and W.-K. Sung (2004). The maximum agreement of two nested phylogenetic networks. In:\ R. Fleischer and G. Trippen (Eds.), {\it 15th International Symposium on Algorithms and Computation}, Lecture Notes in Computer Science, Volume 3341, pp. 581--593.

\bibitem{kelk20}
S. Kelk and S. Linz (2020). New reduction rules for the tree bisection and reconnection distance. {\it Annals of Combinatorics}, 24:475--502.

\bibitem{kelk24}
S. Kelk, S. Linz, and R. Meuwese (2024). Deep kernelization for the tree bisection and reconnection (TBR) distance in phylogenetics, {\it Journal of Computer and System Sciences}, 142:103519.

\bibitem{klawitter18}
J. Klawitter (2018). The SNPR neighbourhood of tree-child networks. {\it Journal of Graph Algorithms and Applications}, 22:329--355.

\bibitem{klawitter19}
J. Klawitter (2019). The agreement distance of rooted phylogenetic networks. {\it Discrete Mathematics and Theoretical Computer Science}, 21:19.

\bibitem{klawitter20}
J. Klawitter (2020). The agreement distance of unrooted phylogenetic networks. {\it Discrete Mathematics and Theoretical Computer Science}, 22:1 \#22.

\bibitem{klawitter-phd}
J. Klawitter (2020). {\it Spaces of phylogenetic networks}. PhD thesis, University of Auckland.

\bibitem{klawitter19b}
J. Klawitter and S. Linz (2019). On the subnet prune and regraft distance. {\it The Electronic Journal of Combinatorics}, 23, Paper 2.3.

\bibitem{kong22}
S. Kong, J. C. Pons, L. Kubatko, and K. Wicke (2022). Classes of explicit phylogenetic networks and their biological and mathematical significance. {\it Journal of Mathematical Biology}, 84:47.

\bibitem{linz09}
S. Linz and C. Semple (2009). Hybridization in nonbinary trees. {\it IEEE/ACM Transactions on Computational Biology and Bioinformatics}, 6:30--45.

\bibitem{linz23}
S. Linz and K. Wicke (2023). Exploring spaces of semi-directed level-$1$ networks. {\it Journal of Mathematical Biology}, 87:70.

\bibitem{markin22}
A. Markin, S. Wagle, T. K. Anderson, and O. Eulenstein (2022). RF-Net $2$: Fast inference of virus reassortment and hybridization networks. {\it Bioinformatics}, 38: 2144--2152.

\bibitem{mueller22}
N. F. M\"uller, K. E. Kistler, and T. Bedford (2022). A Bayesian approach to infer recombination patterns in coronaviruses. {\it Nature Communication}, 13:4186.

\bibitem{mueller20}
N. F. M\"uller, U. Stolz, G. Dudas, T. Stadler, and T. G. Vaughan (2020). Bayesian inference of reassortment networks reveals fitness benefits of reassortment in human influenza viruses. {\it Proceedings of the National Academy of Sciences of the United States of America}, 117:17104--17111.

\bibitem{olver22}
N. Olver, F. Schalekamp, S. van der Ster, L. Stougie, and A. van Zuylen (2022). A duality based $2$-approximation algorithm for maximum agreement forest. {\it Mathematical Programming}, 198:811--853.

\bibitem{stjohn17}
K. St. John (2017). Review paper: The shape of phylogenetic treespace. {\it Systematic Biology}, 66:e83--e94.

\bibitem{semple03}
C. Semple and M. Steel (2003). {\it Phylogenetics}. Oxford University Press.

\bibitem{shi18}
F. Shi, J. Chen, Q. Feng, and J. Wang (2018). A parameterized algorithm for the maximum agreement forest problem on multiple rooted multifurcating trees. {\it Journal of Computer and System Sciences}, 97:28--44.

\bibitem{valiente-book}
G. Valiente (2009). {\it Combinatorial Pattern Matching Algorithms in Computational Biology Using Perl and R}. Chapman \& Hall.

\bibitem{vaniersel22}
L. van Iersel, R. Janssen, M. Jones, Y. Murakami, N. Zeh (2022). A practical fixed-parameter algorithm for constructing tree-child networks from multiple binary trees. {\it Algorithmica}, 84, 917--960.

\bibitem{vanwersch22}
R. van Wersch, S. Kelk, S. Linz, and G. Stamoulis (2022). Reflections on kernelizing and computing unrooted agreement forests. {\it Annals of Operations Research}, 309:425--451.

\bibitem{whidden13}
C. Whidden, R. G. Beiko, and N. Zeh (2013). Fixed-parameter algorithms for maximum agreement forests. {\it SIAM Journal on Computing}, 42:1431--1466.



\end{thebibliography}
\end{document}